\documentclass{amsart}
\usepackage{amssymb}
\usepackage{mathrsfs}
\usepackage[stretch=10,shrink=10]{microtype}
\usepackage{mathtools}
\usepackage[showframe=false, margin = 1.5 in]{geometry}
\usepackage[shortlabels]{enumitem}
\usepackage{tikz}

\usepackage{pgfplots}
\usepgfplotslibrary{external}
\pgfplotsset{compat=1.11}
\usepackage{hyperref}
\usepackage{theoremref}
\makeatletter
\def\thmref@flush{%
   \ifx\thmref@last\empty\else
      \ifthmref@comma, \thmref@finaltrue\fi \thmref@commatrue
      \thmref@last \ifx\thmref@stack\empty\else s\fi \thmref@num 0
      \let\do\thmref@one \thmref@stack
      \ifcase\thmref@num\or\space and\else\thmref@finaltrue, and\fi
      ~\ref{\thmref@head}\let\thmref@stack\empty\fi}
\def\thmref@one#1{\ifnum\thmref@num>0,\fi
   \space\ref{#1}\advance\thmref@num 1\relax}
\makeatother

\newtheorem{thm}{Theorem}[section]
\newtheorem{lemma}[thm]{Lemma}
\newtheorem{prop}[thm]{Proposition}
\newtheorem{cor}[thm]{Corollary}

\theoremstyle{remark}
\newtheorem{remark}[thm]{Remark}

\theoremstyle{definition}

\newtheorem{define}[thm]{Definition}

\newtheorem{example}[thm]{Example}
\newcommand{\E}{\mathbb{E}}
\renewcommand{\P}{\mathbb{P}}
\newcommand{\NN}{\mathbb{N}}
\newcommand{\1}{\mathbf{1}}
\newcommand{\Poi}{\mathrm{Poi}}
\newcommand{\Bin}{\mathrm{Bin}}

\newcommand{\eqd}{\overset{d}{=}}

\newcommand{\Ff}{\mathscr{F}}

\DeclarePairedDelimiter\abs{\lvert}{\rvert}%
\newcommand{\crit}{\mathrm{crit}}
\newcommand{\GW}[1]{\mathrm{GW}_{#1}}

\newcommand{\Tt}{\mathcal{T}}

\newcommand{\Be}[3][x]{B^{=}_{#2,#3}(#1)}
\newcommand{\Bg}[2]{B^{\geq}_{#1,#2}(x)}
\newcommand{\tier}{\mathrm{tier}}
\newcommand{\lambdacrit}{\lambda_{\mathrm{crit}}}

\newcommand{\Biggmid}{\ \Bigg\vert\ }
\newcommand{\sss}{\mathfrak{S}}
\newcommand{\barT}{{\overline{T}}}
\newcommand{\barS}{{\overline{S}}}
\newcommand{\topol}{{\Pi_{\infty}}}

\author{Tobias Johnson}
\address{Department of Mathematics, College of Staten Island}
\email{tobias.johnson@csi.cuny.edu}
\thanks{The author received support from NSF grant DMS-1811952 and PSC-CUNY Award \#62628-00 50.}

\title{Continuous phase transitions on Galton--Watson trees}

\begin{document}

\begin{abstract}
  Distinguishing between continuous and first-order phase transitions is a major challenge
  in random discrete systems.
  We study the topic for events with recursive structure on Galton--Watson
  trees. For example, let $\Tt_1$ be the event that
  a Galton--Watson tree is infinite, and let $\Tt_2$ be the event that it contains an infinite
  binary tree starting from its root. These events satisfy similar recursive properties: 
  $\Tt_1$ holds if and only if $\Tt_1$ holds for at least
  one of the trees initiated by children of the root, and $\Tt_2$ holds if and only if $\Tt_2$ holds for
  at least two of these trees.  
  The probability
  of $\Tt_1$ has a continuous phase transition, increasing from $0$ when the mean of the child
  distribution increases above $1$. On the other hand, the probability of $\Tt_2$ has a first-order
  phase transition, jumping discontinuously to a nonzero value at criticality.
  Given the recursive property satisfied by the event, we describe the critical child distributions
  where a continuous phase transition takes place. In many cases, we also characterize
  the event undergoing the phase transition.
\end{abstract}

\maketitle

\section{Introduction}

Understanding phase transitions is a central task in discrete probability
and statistical physics. 
One of the most basic questions about a phase transition is whether it is \emph{continuous}
or \emph{first-order}. That is, when a quantity undergoes a phase transition, does
it vary continuously as a parameter is varied, 
or does it take a discontinuous jump at criticality? This question is often difficult.
For example, the phase transition for the probability that the origin belongs to an infinite component
in bond percolation on the lattice is thought to be continuous, but it remains unproven
in dimensions $3,\ldots,10$ \cite{HS,FH}.

This paper investigates phase transitions on Galton--Watson trees for events
satisfying certain recursive properties.
This setting is inspired by two examples. Let $\Tt_1$ be the set of infinite rooted trees,
and let $\Tt_2$ be the set of trees containing an infinite binary tree starting from the root.
Let $T_\lambda$ be a Galton--Watson tree with child distribution $\Poi(\lambda)$.
The event $\{T_\lambda\in\Tt_1\}$ has probability~$0$ for $\lambda<1$. It undergoes a continuous
phase transition at $\lambda=1$, with its probability rising from $0$ as $\lambda$ increases above
$1$. On the other hand, the event $\{T_\lambda\in\Tt_2\}$ has probability $0$
for $\lambda<\lambdacrit\approx 3.35$. Its probability jumps to approximately .535 at $\lambda=\lambdacrit$
and increases continuously after that. See \cite[Example~5.5]{JPS} for a detailed treatment
of this example; see \cite{PSW} for this example in the context
of random graphs; and see \cite{Dekking} for an earlier analysis of $\Tt_2$ and proof
that the phase transition is discontinuous for a different family of child distributions.

The sets $\Tt_1$ and $\Tt_2$ both satisfy recursive properties.
A tree is in $\Tt_1$ if and only if the root has at least one child that initiates
a tree in $\Tt_1$. Similarly, a tree is in $\Tt_2$ if and only if the root has at least two children
that initiate trees in $\Tt_2$. Why does $\Tt_1$ have a continuous phase transition
while $\Tt_2$ does not?
The goal of this paper is to answer this question, and more generally
to explain the connection between the recursive property that an event satisfies
and the phase transition that the event undergoes.
It will take some work to state our results, but let us start with an informal account.

First, the event $\Tt_2$ will never have a continuous phase transition
under any family of child distributions. For this event, we say that the \emph{threshold function}
$h(\ell)$ is identically $2$, meaning that regardless of the count $\ell$ of children of the root
of the tree, the event $\Tt_2$ holds if and only if at least two of the children initiate
a tree in $\Tt_2$. For $\Tt_1$, the associated threshold function is identically $1$.
(We will define threshold functions more formally in Section~\ref{subsec:automata.interpretations}.)
\thref{thm:derivatives,thm:criticality} give a criterion for whether a continuous phase transition
occurs at a child distribution $\chi$ given the threshold function $h$
of the event. In particular, a continuous phase transition
can occur at a child distribution $\chi$ only if
\begin{align*}
  \sum_{\text{$\ell\colon h(\ell)=1$}}\chi(\ell)\ell=1.
\end{align*}
This is satisfied for $\Tt_1$ whenever the child distribution has mean~$1$, but it is never satisfied
for $\Tt_2$.

The criterion given by \thref{thm:derivatives,thm:criticality} for when continuous phase transitions
occur is one of the two main results
of the paper, although it is not particularly difficult to show
using results from \cite{JPS}.
The bulk of the work in this paper is to prove the other main result, \thref{thm:critical.event}, 
which runs in the opposite direction as our examples so far.
Suppose we start with a recursive property, without any example of a set of trees
satisfying the property.
\thref{prop:interpretable.crit,thm:derivatives,thm:criticality} work together
to prove that there exists some set of trees satisfying the recursive property,
and that at a certain Galton--Watson measure the probability of this set of trees
undergoes a continuous phase transition. But these results do not describe this set of trees.
 \thref{thm:critical.event} characterizes this set in many circumstances.

To state our results, we must establish what exactly we mean when we say a set 
of trees satisfies a recursive property. A more general version of this framework is given in \cite{JPS}.
Our terminology here is consistent with this more general version,
though we will only introduce what we need here.

\subsection{General notation}\label{subsec:notation}
For a probability distribution $\chi$ on the nonnegative integers, we will
abbreviate quantities like $\chi(\{n\})$ to $\chi(n)$. We use $\GW{\chi}$
to denote the Galton--Watson measure with child distribution $\chi$ on the space
of rooted trees. 
Let $n_t(v)$ denote the number of children of a vertex~$v$ in a rooted tree~$t$.
We refer to the subtrees originated by the children of the root of a tree as
its \emph{root-child subtrees}.
We abuse notation slightly and use expressions like $\Bin(n,p)$ and $\Poi(\mu)$
to denote both a distribution
and a random variable with that distribution, in statements like 
$\P[\Bin(n,p)=k] = \binom{n}{k}p^k(1-p)^{n-k}$. For a random variable $N$ on the nonnegative integers,
$\Bin(N,p)$ denotes a random variable whose law is
the mixture of binomial distributions governed by the law of $N$ (i.e., 
$\P[\Bin(N,p)=k]=\sum_{n=0}^{\infty}\P[N=n]\P[\Bin(n,p)=k]$). We denote the falling
factorial $n(n-1)\cdots (n-k+1)$ by the notation $(n)_k$. Let $\NN=\{0,1,2,\ldots\}$.


\subsection{Encoding recursive properties}\label{subsec:automata.interpretations}

As we hinted earlier, we will describe recursive properties by giving a \emph{threshold function}
$h$. For a tree whose root has $\ell\geq 0$ children, we think of $h(\ell)$ as the minimum
number of its root-child subtrees with a given property to force the tree
itself to have that property. If an event is consistent with the recursive property encoded
by $h$, we call it an \emph{interpretation} of $h$. To formalize this, for a rooted tree $t$
let $\ell(t)$ denote the number of children of the root of $t$.
For a set of trees $\Tt$, 
let $c(t,\Tt)$ denote the number of root-child subtrees of $t$ that are elements of $\Tt$. 
For a given child distribution $\chi$
and threshold function $h$, we say that a $\GW{\chi}$-measurable set of trees $\Tt$ is an
\emph{interpretation of $(\chi,h)$} if
\begin{align}\label{eq:interpretation.def}
  t\in\Tt \iff c(t,\Tt)\geq h\bigl(\ell(t)\bigr) \qquad \text{for $\GW{\chi}$-a.e.\ rooted tree $t$.}
\end{align}
For example, the set $\Tt_1$ of infinite trees is an interpretation of $(\chi,h)$
where $h(\ell)\equiv 1$ and the set $\Tt_2$ of trees containing an infinite binary tree starting from
the root is an interpretation of $(\chi,h)$ where $h(\ell)\equiv 2$.
In both of these cases, \eqref{eq:interpretation.def} holds for all trees $t$, not just
for $\GW{\chi}$-a.e.\ tree $t$, which renders $\chi$ irrelevant.
In such cases we will often call our event \emph{an interpretation of $h$}, omitting
reference to the child distribution. (Excluding negligible sets in the definition is required for some
of the results in \cite{JPS}.)

For context, let us describe how these recursive properties
fit into the broader class considered in \cite{JPS}.
Suppose that the root of $t$ has $\ell$ children, and 
that $n_1$ originate trees with some property while $n_0$ of them do not.
Suppose that the counts $n_0$ and $n_1$ determine whether $t$ itself has the property,
and let the map $A\colon\NN^2\to\{0,1\}$ specify this, with $A(n_0,n_1)=1$ when $t$
has the the property and $A(n_0,n_1)=0$ when it does not. In \cite{JPS}, this map $A$ is
called a \emph{tree automaton}, and an event consistent with the recursive property described
by the automaton is called an interpretation of it. Recursive properties defined
by a threshold function $h$ correspond to
automata of the form $A(n_0,n_1)=\1\{n_1\geq h(n_0+n_1)\}$, which in \cite{JPS}
are called \emph{monotone automata}. We restrict ourselves to such automata because we have
stronger results for them, primarily because the Margulis--Russo lemma provides
a powerful tool for their analysis (see \cite[Section~5.1]{JPS}).

The tree automata described above are called \emph{two-state}, in that each tree has one
of two possible states (having the property or not having the property) and the state of a tree
is determined by the states of its root-child subtrees. In \cite{JPS}, automata are considered
with more than two states, which again increases the complexity of the theory.

In this paper, we will often impose the additional condition that $h(\ell)$ is 
(nonstrictly) increasing in $\ell$. This is also a form of monotonicity for the recursive property;
it amounts to declaring that if a tree $t$ has the property, then it still has it
after attaching an additional subtree to the root.

\subsection{Fixed points}\label{subsec:fixed}
As we will soon see, the probability of an interpretation under the Galton--Watson measure
satisfies a fixed-point equation determined by the threshold function
and child distribution.
The classical example is the probability that a Galton--Watson
survives
(i.e., probability of the set $\Tt_1$ discussed in Section~\ref{subsec:automata.interpretations}).
Taking $T_\lambda\sim\GW{\Poi(\lambda)}$,
let $x= \P[T_\lambda\in\Tt_1]$. Since $T_\lambda\in\Tt_1$ if and only if at least one of the root-child
subtrees of $T_\lambda$ is in $\Tt_1$, and each of the $\Poi(\lambda)$ 
root-child subtrees has probability~$x$ of being in $\Tt_1$,
\begin{align}\label{eq:gw.classic}
  x = \P[\Poi(\lambda x)\geq 1] = 1 - e^{-\lambda x}
\end{align}
by Poisson thinning.
This equation has two solutions when $\lambda>1$, and in this case $x$ turns out to be
the larger of the two (the smaller is $0$).

To give the fixed-point equation in a general case, we define the \emph{automaton distribution
map} $\Psi(x)$ for a given child distribution $\chi$ and threshold function $h$.
(The terminology \emph{automaton distribution map} comes from a generalization in \cite{JPS}
that maps distributions to distributions.)
With $L\sim\chi$, we define
\begin{align}\label{eq:monotone.Psi}
  \Psi(x) = \P[ \Bin(L,x) \geq h(L) ] = \sum_{\ell=0}^{\infty}\chi(\ell)\P\bigl[\Bin(\ell,x)\geq h(\ell)\bigr].
\end{align}
In words, $\Psi(x)$ is the probability that at least $h(L)$ out of $L$ root-child subtrees of a Galton--Watson
tree have some property that holds for each of them with probability~$x$.

Observe that the right-hand side of \eqref{eq:gw.classic} is the automaton distribution map
for $\chi=\Poi(\lambda)$ and $h(\ell)\equiv 1$. Thus \eqref{eq:gw.classic} is the statement
that the probability of the interpretation $\Tt_1$ under $\GW{\chi}$
is a fixed point of the automaton distribution map.
In fact, it holds in general that for any child distribution $\chi$ and threshold function $h$, 
the probability of an interpretation $\Tt$ of $(\chi,h)$ is a fixed point
of its automaton distribution map: Let $T\sim\GW{\chi}$ and let $x=\P[T\in\Tt]$.
Conditional on the number of children of the root of $T$,
each root-child subtree in $T$ lies in $\Tt$ independently with probability~$x$, since
the root-child subtrees are themselves independently sampled from $\GW{\chi}$.
Because $\Tt$ is an interpretation of $A$, the tree's membership in $\Tt$ is determined
from its root-child subtrees' membership in $\Tt$ according to $h$.
Thus $\Psi(x)=\P[T\in\Tt]=x$.

We are interested in circumstances in which the automaton distribution maps have $0$ as a fixed point,
since we are investigating phase transitions emerging from $0$.
Thus we typically require the threshold function $h$ to satisfy
$h(\ell)\geq 1$ for all $\ell\geq 0$.
(Strictly speaking, to make $0$ a fixed point 
we only need $h(\ell)\geq 1$ for $\ell$ in the support of the child distribution,
but the value of $h(\ell)$ for $\ell$ outside of this support is irrelevant anyhow.)

\subsection{Results}
\pgfplotscreateplotcyclelist{colorchange}{%
blue,every mark/.append style={solid,fill=blue!80!black},mark=*\\%
red,densely dashed,every mark/.append style={solid,fill=red!80!black},mark=*\\%
brown!60!black,densely dotted,every mark/.append style={solid,fill=brown!80!black},mark=*\\%
black,dashdotted,mark=*,every mark/.append style={solid}\\%
blue,densely dotted,every mark/.append style={solid,fill=blue!80!black},mark=*\\%
red,every mark/.append style={solid,fill=red!80!black},mark=*\\%
brown!60!black,densely dashed,every mark/.append style={
solid,fill=brown!80!black},mark=square*\\%
black,densely dashed,every mark/.append style={solid,fill=gray},mark=otimes*\\%
blue,densely dashed,mark=star,every mark/.append style=solid\\%
red,densely dashed,every mark/.append style={solid,fill=red!80!black},mark=diamond*\\%
}  
\begin{figure}
\begin{tikzpicture}
    \begin{axis}[xmin=0, xmax=1,
                 legend pos=outer north east, thick, xlabel={$x$}, ylabel={$\Psi_t(x)-x$},
                 legend cell align=left,cycle list name=colorchange]
      \addplot[thin,black, forget plot, no markers] {0};
      \addplot+[mark indices={1001}] table {t4.dat};
      \addplot+[mark indices={819}] table {t3.dat};
      \addplot+[mark indices={563}] table {t2.dat};
      \addplot+[mark indices={347}] table {t1.dat};
      \addplot+[no markers] table {t0.dat};
      \addlegendentry{$t=\frac13$}
      \legend{$t=\textstyle\frac13$,$t=.2$,$t=.1$,$t=.05$,$t=0$}

    \end{axis}
    
  \end{tikzpicture}
  
  \caption{Let $\chi_t$ be the probability measure placing the
  vector of probabilities $\bigl(\frac13-t,\,0,\,\frac12,\,\frac16+t\bigr)$ on values
  $0,\,1,\,2,\,3$. Let $h(0)=h(1)=h(2)=1$ and $h(3)=2$.
  The graphs above depict $\Psi_t(x)$, the automaton distribution map of $(\chi_t,h)$.
  The recursive tree system $(\chi_t,h)$ is critical at $t=0$ in the sense of
  \thref{def:critical}. As $t$ increases,
  a single interpretable fixed
  point emerges and increases to $1$ as $t$ rises to $1/3$.
  }\label{fig:emergence}
\end{figure}
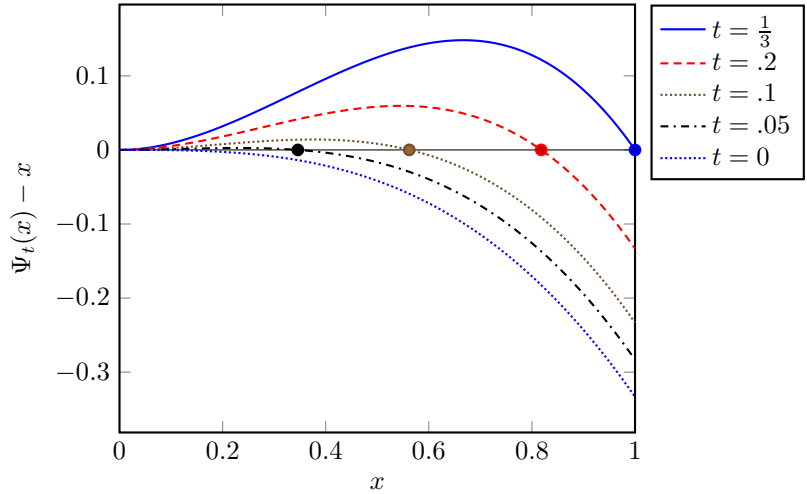

Fix a threshold function $h$, and let $\Psi_{\chi}$ be the automaton distribution map for
$(\chi,h)$.
Our aim is to understand the circumstances in which $\chi$ is \emph{critical}, 
in the sense that there is an event satisfying the recursive property encoded by $h$
whose probability emerges from $0$ as $\chi$ is perturbed (we will make this
definition precise in \thref{def:critical}).
Since an event satisfying
the recursive property has probability given by a fixed point of $\Psi_\chi$,
a new fixed point must emerge from $0$ as $\chi$ is perturbed.
Based on the idea that $\Psi_\chi$ changes continuously in $\chi$, 
intuition suggests that $\Psi'_\chi(0)=1$ is necessary in order to have a fixed point emerge from $0$
as $\chi$ is varied (see Figure~\ref{fig:emergence} for an example of a fixed point emerging).
This thought is on the right track, but it raises some questions:
\begin{enumerate}[(a)]
  \item Suppose that $\chi$ can be perturbed so that a fixed
    point of $\Psi_\chi$ emerges from $0$. Is it always the case that this fixed point has an interpretation?
    That is, is there an event satisfying the recursive property whose
    probability is given by the fixed point (and which therefore has a phase transition)?
    \label{i:comp1}
  \item Can we characterize the critical child measures $\chi$ in a more direct way than 
    stating properties of $\Psi_\chi$?
    \label{i:comp2}
  \item Suppose that $\Psi_\chi$ has a fixed point emerging from $0$ as $\chi$ is perturbed,
    and we can determine that indeed $\chi$ is critical, i.e., that
    there exists an interpretation associated with this fixed point undergoing
    a phase transition. Can we state what the interpretation is in any satisfying way?
        \label{i:comp3}
\end{enumerate}

Before we address these questions and present our results, 
let us recall and define some notation.
For a given child distribution $\chi$, we will be posing questions
about the interpretations of $(\chi,h)$, as defined in Section~\ref{subsec:automata.interpretations}.
We call $(\chi,h)$ a \emph{recursive tree system}, and
we take as part of the definition that $h(\ell)\geq 1$ for all $\ell\geq 0$.
As we explained in Section~\ref{subsec:fixed}, any interpretation of $(\chi,h)$ has $\GW{\chi}$-measure satisfying the fixed-point equation
$\Psi(x)=x$, where $\Psi$ is the automaton distribution map of $(\chi,h)$.
If $\Tt$ is an interpretation of $(\chi,h)$ with $\GW{\chi}(\Tt)=x_0$, then we say that
$\Tt$ is the interpretation associated with the fixed point $x_0$ (we write \emph{the interpretation}
rather than \emph{an interpretation} because we show in \thref{prop:interpretable.crit} that a given
fixed point can have at most one interpretation).
We refer to the fixed points of $\Psi$ as the fixed points of $(\chi,h)$.
For a system $(\chi,h)$, we define its \emph{$k$th tier} as the set of 
values $\ell\geq 1$ in the support of $\chi$ with $h(\ell)=k$.
That is, tier~$k$ for $(\chi,h)$ is defined as
\begin{align}\label{eq:tier.def}
  \tier(k)=\tier_{\chi,h}(k)=\bigl\{ \ell\geq 1\colon \text{$h(\ell)=k$ and $\chi(\ell)>0$} \bigr\}.
\end{align}

Question~\ref{i:comp1} is resolved by the following criterion for when a fixed point
of $(\chi,h)$ admits an interpretation:

\begin{prop}\thlabel{prop:interpretable.crit}
  Let $\chi$ have finite expectation and more than one point of support, 
  let $\Psi$ be the automaton distribution map
  of the recursive tree system $(\chi,h)$, and let $0<x_0<1$ be a fixed point of
  $(\chi,h)$.
  There exists an interpretation of $(\chi,h)$ associated with $x_0$
  if and only if $\Psi'(x_0)\leq 1$.
  When an interpretation of $x_0$ exists, it is unique up to $\GW{\chi}$-negligible sets.
\end{prop}
We give a proof in Section~\ref{sec:analytic}, though it just amounts to tying together
results from \cite{JPS}. 
\thref{prop:interpretable.crit} does not address the case of $x_0=0$ or $x_0=1$
because such fixed points have trivial interpretations associated with them,
namely the empty set in the case of $0$ and the set of all rooted
trees in the case of $1$.

To address question~\ref{i:comp2}, we give a formula for the derivatives of $\Psi$
at zero in terms of $\chi$.
The notation $(\ell)_m$ in \eqref{eq:derivatives1} and \eqref{eq:derivatives2}
 denotes the falling factorial $\ell(\ell-1)\cdots (\ell-m+1)$.
\begin{thm}\thlabel{thm:derivatives}
  Let $\chi$ be a child distribution with finite $m$th moment, and let $\Psi$
  be the automaton distribution map of
  the recursive tree system $(\chi,h)$. Then for $m\geq 1$,
  \begin{align}\label{eq:derivatives1}
    \Psi^{(m)}(0) &= \sum_{\ell=m}^{\infty}(-1)^{m+h(\ell)}\binom{m-1}{h(\ell)-1}\chi(\ell)(\ell)_m\\
    &=
    \sum_{j=1}^m (-1)^{m+j}\binom{m-1}{j-1}\sum_{\ell\in\tier(j)} \chi(\ell)(\ell)_m.\label{eq:derivatives2}
  \end{align}
\end{thm}
We highlight the $m=1,2$ cases of this theorem:
\begin{align}
  \Psi'(0) &= \sum_{\ell\in\tier(1)} \chi(\ell)\ell,\label{eq:Psi'}\\
  \Psi''(0) &= \sum_{\ell\in\tier(2)} \chi(\ell)\ell(\ell-1)
    -\sum_{\ell\in\tier(1)} \chi(\ell)\ell(\ell-1) .\label{eq:Psi''}
\end{align}

It follows from \thref{thm:derivatives} that the value of $\Psi^{(m)}(0)$ depends
only on the mass that $\chi$ places on the first $m$ tiers.  We
state this formally since we will often use it:
\begin{cor}\thlabel{rmk:only.tierm}
  Assume that $\chi$ and $\widetilde\chi$ have finite $m$th moments.
  Let $\Psi$ and $\widetilde\Psi$ be the automaton distribution
  maps of the recursive tree systems $(\chi,h)$ and $(\widetilde\chi,h)$, respectively.
  If $\chi(\ell)=\widetilde\chi(\ell)$ whenever $h(\ell)\leq m$, then
  $\Psi^{(m)}(0)=\widetilde\Psi^{(m)}(0)$.
\end{cor}

Now, we can relate conditions on the derivatives of $\Psi$ at $0$ back to the child distribution.
With this in mind, we will state our criteria
for where continuous phase transitions occur. 
Let $\topol$ denote the space of probability measures on the nonnegative
integers with all moments finite. On $\topol$,
for any $n\geq 0$ we can define a metric 
\begin{align*}
  d_n(\chi_1,\chi_2) &= \sum_{k=1 }^{\infty} k^n\abs{\chi_1(k)-\chi_2(k)}.
\end{align*}
We topologize $\topol$ by declaring that $\chi_n\to\chi$ if $d_n(\chi_n,\chi)\to 0$ for all $n\geq 0$.
We work in this space to avoid pathologies; see \thref{rmk:topology} for more details.

\begin{define}\thlabel{def:critical}
For a recursive tree system $(\chi,h)$ with $\chi\in\topol$, we say that $(\chi, h)$ is \emph{critical} 
if for any $\epsilon>0$, all neighborhoods
of $\chi$ in $\topol$ contain a measure $\pi$ such that there is an interpretation $\Tt$ 
of $(\pi,h)$
satisfying $0<\GW{\pi}(\Tt)<\epsilon$.
Equivalently, $(\chi,h)$ is critical if there exists a sequence
$\chi_n\in\topol$ converging to $\chi$ such that $(\chi_n,h)$ has an interpretation
$\Tt_n$ with $\GW{\chi_n}(\Tt_n)\searrow 0$. 
\end{define}

\begin{thm}\thlabel{thm:criticality}
  Let $\chi\in\topol$ have more than one point of support.
  The recursive tree system $(\chi,h)$ with automaton distribution map $\Psi$
  is critical if and only if $\Psi'(0)=1$ and $\Psi''(0)\leq 0$.
\end{thm}

\begin{remark}\thlabel{rmk:critical.def}
  One might object that to correctly capture
the idea of a phase transition, we should insist on a single interpretation $\Tt$
satisfying $\GW{\chi_n}(\Tt)\searrow 0$, rather than a sequence of interpretations
$\Tt_n$ with $\GW{\chi_n}(\Tt_n)\searrow 0$. For example, suppose $h(\ell)\equiv 1$, $\chi=\Poi(1)$,
and $\chi_n=\Poi(1+1/n)$. Then for the set of infinite trees $\Tt_1$, we have
$\GW{\chi_n}(\Tt_1)\searrow 0$.
In fact, this more stringent requirement is equivalent to our original one, as we now show.
First, we claim that for different child distributions $\chi$ and $\chi'$, the measures
$\GW{\chi}$ and $\GW{\chi'}$ restricted to infinite trees are mutually singular.
To see this, observe that for $\GW{\chi}$-a.e.\ infinite tree $t$, the empirical distribution
of the numbers of children of the vertices at level~$n$ of the tree converges to $\chi$.
Hence the supports of $\GW{\chi}$ and $\GW{\chi'}$ on infinite trees are disjoint.

Next, any interpretation contains only infinite trees by our requirement that $h(\ell)\geq 1$.
To see this, let $\Tt$ be an interpretation, and observe that a single-vertex tree is not a member
of $\Tt$ since 0 of its 0 root-child vertices are in $\Tt$, and $h(0)\geq 1$.
Then since single-vertex trees are not members of $\Tt$, no height-$1$ tree can be in $\Tt$,
and hence no height-$2$ tree can be in $\Tt$, and so on.

Thus, if we have a sequence of interpretations $\Tt_n$ of $(\chi_n,h)$ satisfying
$\GW{\chi_n}(\Tt_n)\searrow 0$, we can stitch them together into a single interpretation
$\Tt$ defined to be equal to $\Tt_n$ on the support of $\GW{\chi_n}$.
However, we will see in \thref{thm:critical.event} that
for a large class of phase transitions, we can define a single
interpretation $\Tt$ in a more satisfying way so that $\GW{\chi_n}(\Tt)\searrow 0$.
\end{remark}

\begin{remark}\thlabel{rmk:topology}
  The details of the topology on $\topol$ are not particularly important to this paper,
  but let us define it in more detail and explain what goes wrong with a looser sense
  of convergence. To make it so that $\chi_n\to\chi$ if and only if $d_n(\chi_n,\chi)\to 0$
  for all $n\geq 0$, consider the product space $\prod_{n=0}^{\infty}\topol$ where
  the $n$th copy of $\topol$ is taken as the metric space $(\topol,d_n)$.
  Now consider the map $\iota\colon \topol\to\prod_{n=0}^{\infty}\topol$ given
  by $\iota(\chi)=(\chi,\chi,\ldots)$. We assign $\topol$ the topology
  induced by $\iota$, i.e., the one formed by pullbacks of open sets in the product space.
  
  Problems arise if we use a coarser topology on $\topol$. For example, suppose we use
  the metric $d_0$, which in this space metrizes the topology of convergence in law.
  Now, for the threshold function $h(\ell)\equiv 1$, even the measure $\delta_0$
  is critical. Indeed, define 
  \begin{align*}
    \chi_n=(1-2/n)\delta_0 +  (2/n)\delta_n.
  \end{align*}
  Then $d_0(\chi_n,\chi)\to 0$ and $\GW{\chi_n}(\Tt_1)\searrow 0$, where $\Tt_1$
  is the set of infinite rooted trees.
\end{remark}

Finally, we address question~\ref{i:comp3} and try to describe the event undergoing the continuous
phase transition. Our result, \thref{thm:critical.event},
characterizes the interpretation associated with the smallest
nonzero fixed point when its automaton distribution map $\Psi(x)$ satisfies 
$\Psi(x)>x$ on some interval $(0,\epsilon)$ for $\epsilon>0$.
This means that the result describes the event undergoing a phase transition
so long as the graph of the automaton distribution map rises above the line $y=x$ as
the phase transition occurs. This occurs in the phase transitions illustrated in
Figures~\ref{fig:emergence} and \ref{fig:phase.transition.classes.good}, but not
in the phase transition shown in Figure~\ref{fig:phase.transition.classes.bad}.
We also mention that \thref{thm:critical.event} requires $h(\ell)$ to be increasing.

\begin{figure}
  \begin{tikzpicture}
    \begin{axis}[xmin=0, xmax=1,
                 legend pos=outer north east, thick, xlabel={$x$}, ylabel={$\Psi_t(x)-x$},
                 legend cell align=left,scaled ticks=false,y tick label style={/pgf/number format/fixed},
                 cycle list name=colorchange]
      \addplot[thin,black, forget plot, no markers] {0};
      \addplot+[mark indices={55,132,251}] table {re5.dat};
      \addplot+[mark indices={42,144,248}] table {re4.dat};
      \addplot+[mark indices={31,155,245}] table {re3.dat};
      \addplot+[mark indices={21,164,242}] table {re2.dat};
      \addplot+[mark indices={11,174,238}] table {re1.dat};
      \addplot+[mark indices={184,233}] table {re0.dat};
      \legend{$t=.05$,$t=.04$,$t=.03$,$t=.02$,$t=.01$,$t=0$}

    \end{axis}
    
  \end{tikzpicture}
  \caption{Graphs of $\Psi_t(x)-x$, where $\Psi_t$ is the automaton distribution map
  of $(\chi_t,h)$ with  $\chi_t = \bigl(\frac{1}{20} - t\bigr)\delta_0 +
    \bigl(\frac12+t\bigr)\delta_2 + \frac{9}{20}\delta_5$, and $h(0)=h(2)=1$ and $h(5)=4$.
    The system $(\chi_t,h)$ is critical at $t=0$ in the sense of \thref{def:critical}, and we can
    see a fixed point emerging from $0$ as $t$ increases. Because $\Psi_t(x)\geq x$
    in a neighborhood of $0$ for $t>0$, the interpretation associated with this fixed point
    is described in \thref{thm:critical.event}. See
    \thref{ex:good.example} for more details.
    }\label{fig:phase.transition.classes.good}
\end{figure}
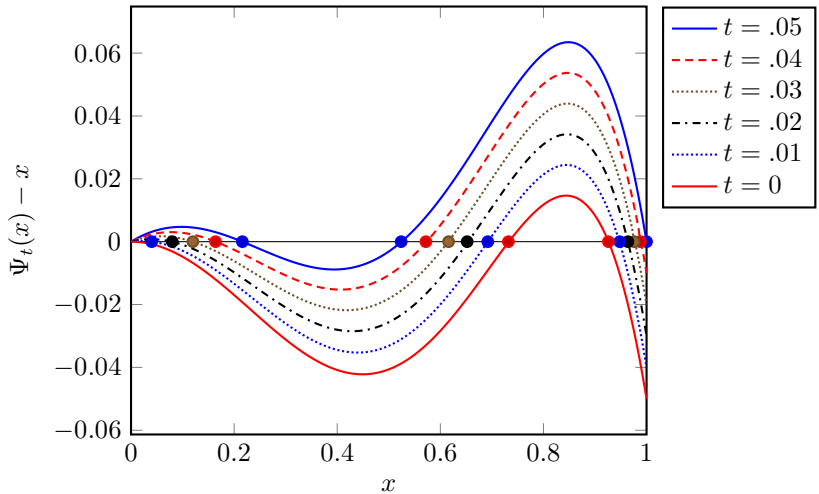

\begin{figure}
  \begin{tikzpicture}[trim axis left]
    \begin{axis}[xmin=0, xmax=1,
                 legend pos=outer north east, thick, xlabel={$x$}, ylabel={$\Psi_t(x)-x$},
                 legend cell align=left,scaled ticks=false,y tick label style={/pgf/number format/fixed},
                           height=241pt,cycle list name=colorchange]
      \addplot[thin,black, forget plot, no markers] {0};
      \addplot+[no markers] table {fr01.dat};
      \addplot+[no markers] table {fr0075.dat};
      \addplot+[no markers] table {fr005.dat};
      \addplot+[no markers] table {fr0.dat};
      \legend{$t=.01$,$t=.0075$,$t=.005$,$t=0$}
    \end{axis}  
  \end{tikzpicture}

  \begin{tikzpicture}[trim axis left]
        \begin{axis}[xmin=0, xmax=6,
                 legend pos=outer north east, thick, xlabel={$x$}, ylabel={$\Psi_t(x)-x$},
                 legend cell align=left,scaled x ticks={real:100},scaled y ticks={real:1},
                 xtick scale label code/.code={},
                 ytick scale label code/.code={${}\cdot 10^{-5}$},
                 x tick label style={/pgf/number format/fixed},y tick label style={/pgf/number format/fixed},
                 height=241pt,cycle list name=colorchange]
      \addplot[thin,black, forget plot,no markers,domain=0:6] {0};
      \addplot+[mark indices={130,238}] table {fr01zoom.dat};
      \addplot+[mark indices={97,181}] table {fr0075zoom.dat};
      \addplot+[mark indices={65,123}] table {fr005zoom.dat};
      \addplot+[no markers] table {fr0zoom.dat};
      \legend{$t=.01$,$t=.0075$,$t=.005$,$t=0$}
    \end{axis}
  \end{tikzpicture}

  \caption{Graphs of $\Psi_t(x)-x$ illustrating a continuous phase transition not satisfying
  the conditions of \thref{thm:critical.event}. Here $\Psi_t$ is the automaton distribution map
  for $(\chi_t,h)$ where $\chi_t = \frac{1}{24}\delta_0 + \bigl(\frac12 - 3t^2\bigr)\delta_2
   + \bigl(\frac16 + t\bigr)\delta_3 + \bigl(\frac{7}{24}+3t^2-t\bigr)\delta_6$,
   and $h(0)=h(2)=1$, $h(3)=2$, and $h(6)=5$. The top plot is our standard view of
   $\Psi_t(x)-x$ as in Figures~\ref{fig:emergence} and \ref{fig:phase.transition.classes.good}.
   In the bottom plot, we zoom in around $x=0$ and see that two fixed points emerge from zero
   as $t$ increases. By \thref{prop:interpretable.crit}, 
   the first fixed point for each system has no interpretation
   but the second one does. But we cannot apply \thref{thm:critical.event}
   to characterize this interpretation. See Section~\ref{sec:questions} for
   further discussion.
  }
    \label{fig:phase.transition.classes.bad}
\end{figure}
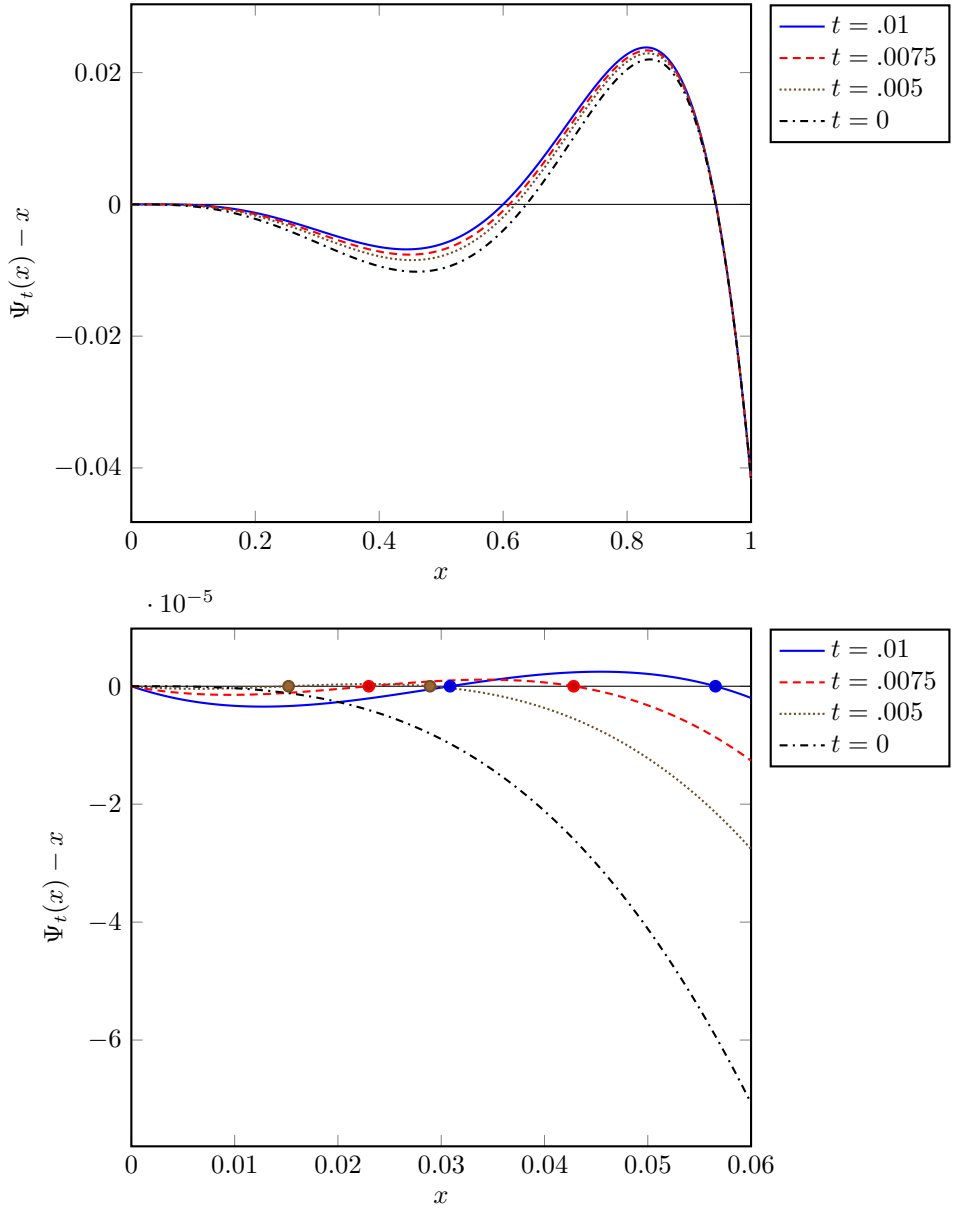

The characterization of the interpretation depends on the behavior of
the automaton distribution map near $0$. We define some terminology about this now.
For $m\geq 1$, we say that the recursive tree system $(\chi,h)$ is \emph{$m$-concordant} 
if the first $m$ derivatives of $\Psi$ at 0 match those of the function $x$. That is,
$(\chi,h)$ is $m$-concordant if $\Psi'(0)=1$ and $\Psi^{(k)}(0)=0$ for $2\leq k\leq m$.
For $m\geq 2$, we say that $(\chi,h)$ is $m$-subcordant (resp.\ $m$-supercordant)
if it is $(m-1)$-concordant and $\Psi^{(m)}(0)<0$ (resp.\ $\Psi^{(m)}(0)>0$).
We say that $(\chi,h)$ is $1$-concordant, $1$-subcordant, or $1$-supercordant 
if $\Psi'(0)=1$, $\Psi'(0)<1$, or $\Psi'(0)>1$, respectively.
When $h$ is clear from context, we will abuse notation and refer to $\chi$ itself
as being $m$-concordant, $m$-subcordant, or $m$-supercordant.
Note that assuming smoothness of $\Psi$ (which holds for $\chi\in\topol$
by \thref{lem:derivatives.converge}),
if $\Psi(x)> x$ holds on some interval $(0,\epsilon)$, then $(\chi,h)$ is $m$-supercordant
for some $m\geq 1$.

Finally, we define the notion of an admissible subtree.
We say that a subtree~$s$ of a rooted tree~$t$ is \emph{admissible} with
respect to a threshold function $h$ if $s$ contains the root of $t$
and $n_s(v)\geq h(n_t(v))$ for all vertices $v\in s$.
We can think of an admissible subtree as a sort of witness to an interpretation.
For example, consider $h(\ell)\equiv 1$, the threshold function encoding a property
that holds for a tree if and only if it holds for at least one of the tree's root-child subtrees.
As we mentioned earlier, this recursive tree
system has two fixed points when $\chi$ has mean greater than $1$, and the interpretation
of the nonzero fixed point is survival of the Galton--Watson tree.
A subtree $S$ of the Galton--Watson tree $T$ is admissible
if and only if $S$ has no leaves (i.e., $n_S(v)\geq 1$ for all $v\in S$).
An admissible subtree thus serves as a witness to 
the Galton--Watson tree being infinite. 
Indeed, the event of $T$ being infinite could equally well be described
as $T$ having an admissible subtree (see \thref{lem:highest.fp} for a generalization).

\begin{thm}\thlabel{thm:critical.event}
  Let $\chi\in\topol$ have more than one point of support. 
  Consider the recursive tree system $(\chi,h)$ with automaton
  distribution map $\Psi$, and assume that $h(\ell)$ is increasing in $\ell$. 
  Suppose that $(\chi,h)$ is $m$-supercordant, and let $x_0$ be the
  smallest nonzero fixed point of $\Psi$.
  Then $x_0$ is interpretable, and its associated interpretation
  is the event that
  $T\sim\GW{\chi}$ contains an admissible subtree $S$ in which all but finitely many
  vertices $v$ satisfy $n_S(v)\leq m$.
\end{thm}

Though it falls outside this narrative of understanding phase transitions, we
mention that in all cases, the \emph{highest} fixed point of a recursive tree system
has a similar characterization:

\begin{prop}\thlabel{lem:highest.fp}
  Let $\chi\in\topol$ have more than one point of support. 
  Consider the recursive tree system $(\chi,h)$ with automaton
  distribution map $\Psi$.
  Let $x_1$ be the
  largest fixed point of $\Psi$.
  Then $x_1$ is interpretable,
  and its associated interpretation is the event that
  $T\sim\GW{\chi}$ contains
  an admissible subtree.
\end{prop}

We close the section with an example of a family of systems $(\chi_t,h)$ undergoing
a continuous phase transition, shown in Figure~\ref{fig:phase.transition.classes.good}.
\begin{example}\thlabel{ex:good.example}
  Define $\chi_t$ and $h$ by
  \begin{align*}
  \chi_t(\ell)=\begin{cases}
    1/20 - t & \text{for $\ell=0$,}\\
    1/2 + t & \text{for $\ell=2$,}\\
    9/20 & \text{for $\ell=5$,}\\
  \end{cases}\qquad\qquad \text{and} \qquad\qquad
  h(\ell)=\begin{cases}
      1 & \text{for $\ell=0$,}\\
      1 & \text{for $\ell=2$,}\\
      4 & \text{for $\ell=5$,}
  \end{cases}
      \end{align*}
  and let $\Psi_t$ be the automaton distribution map of the recursive tree system $(\chi_t,h)$.
  By \thref{thm:derivatives},
  \begin{align*}
    \Psi'_t(0) &= 2\chi_t(2) = 1 + 2t,\\
    \Psi''_t(0) &= -2\chi_t(2) = -1-2t.
  \end{align*}
  The system $(\chi_0,h)$ is critical by \thref{thm:criticality},
  since $\Psi'_0(0)=1$ and $\Psi''_0(0)=-1$.
  For $t>0$, the system is $1$-supercordant (i.e., $\Psi'_t(0)>1$).
  
  In Figure~\ref{fig:phase.transition.classes.good}, we show the graphs of $\Psi_t(x)-x$,
  so that fixed points of $\Psi_t$ appear as roots. 
  At $t=0$, the system has two nonzero fixed points, at $x\approx.73$ and $x\approx.93$.
  As $t$ grows,
  a fixed point $x_0(t)$ emerges from $0$. 
  We have $\Psi'_t(x_0(t))<1$, evident from the graph of $\Psi_t(x)-x$ where $x_0$ is a down-crossing
  root.
  By \thref{prop:interpretable.crit}, the fixed point $x_0(t)$ is interpretable.
  By \thref{thm:critical.event}, the interpretation of $(\chi_t,h)$
  associated with $x_0(t)$ is the event~$\Tt_0$ that
  $T\sim\GW{\chi}$ has an admissible subtree $S$ in which all but finitely many vertices
  $v\in S$ satisfy $n_S(v)\leq 1$. Since $h(\ell)>0$ for all $\ell\geq 0$, an admissible subtree
  has no leaves, and thus all but finitely many vertices have $n_S(v)=1$.

  Recall that $S\subseteq T$ is admissible if it contains the root of $T$ and for each $v$ in $S$,
  we have $n_S(v)\geq h(n_T(v))$. In this case, if a vertex~$v$ has $5$ children in $T$,
  it can only be in $S$ if at least $4$ of those children are also in $S$;
  if $v$ has $2$ children in $T$, it can only be in $S$ if at least one of those children
  is in $S$; and if it has no children in $T$, it cannot be in $S$. Thus on the event
  $\Tt_0$, the tree $T$ contains an admissible subtree where all but finitely many vertices
  have $2$ children in $T$.
  
  We can see directly that $\GW{\chi_0}(\Tt_0)=0$ by observing that the subtree of
  $T\sim\GW{\chi_0}$ consisting only of the vertices with $2$ or fewer children forms a critical
  Galton--Watson tree (its child distribution is $\frac12\delta_0+\frac12\delta_2$).
  Thus it has no chance of being infinite. Consequently, for any vertex~$v$ in $T$, 
  there is no chance that $T(v)$ contains an admissible subtree consisting of only vertices
  with $2$ children in $T$.
  
  Viewing the graph in Figure~\ref{fig:phase.transition.classes.good}, we observe that
  $\Psi_t$ has derivative greater than $1$ at its middle fixed point, which means
  that it has no interpretation by \thref{prop:interpretable.crit}. The largest fixed point
  has the interpretation that $T\sim\GW{\chi_t}$ contains an admissible subtree,
  by \thref{lem:highest.fp}.
\end{example}

\subsection{Related work}

A finite random structure like a random graph will not typically experience a true phase transition.
The analogous concept in this area is a \emph{sharp threshold} for some property, meaning that
the property holds with probability that transitions from $0$ to $1$ in a parameter
window that tends to $0$ as the system grows.
The motivating phase transitions of this paper---the 
continuous phase transition for survival and the first-order 
transition for existence of a binary subtree in a Galton--Watson tree---have analogues
for Erd\H{o}s--R\'enyi random graphs in this sense: the existence of giant component \cite[Chapter~11]{AS},
and the existence of a $3$-core \cite{PSW,Riordan}. The first of these examples is essentially
a continuous phase transition while the second is essentially first-order. For example, when one reaches
the threshold for a $3$-core to exist, it immediately makes up a positive linear fraction of the graph's
vertices. These examples have been studied extensively, though not in the sort of general framework
considered in this paper.

As for more general studies of phase transitions and sharp thresholds, in finite systems
there is a line of inquiry centered on giving
conditions for a property to have a sharp threshold \cite{SS,LS,FK,F}. Many of these results
use the theory of Boolean functions and the Margulis--Russo formula (see
\cite{GS} for background), also used in this paper for the proof
of \thref{prop:interpretable.crit} via \cite{JPS}.
These results on sharp thresholds for finite random structures have been applied in impressive
ways to prove results on phase transitions for infinite systems \cite{BR,DC}.
There is also considerable nonrigorous literature by physicists on distinguishing between
continuous and first-order phase transitions (see for example \cite{explosive,EHdOF}).

For Galton--Watson trees specifically, Podder and Spencer investigate probabilities
of events that can be described in first-order logic (no connection to first-order phase transitions)
in \cite{PS1,PS2}. These events have the same sort of recursive description as the events
considered here, generally with more than two states. But their fundamental result
\cite[Theorem~1.2]{PS2} is that these events never undergo phase transitions at all.
In \cite{HM}, Holroyd and Martin consider various two-player games whose moves are modeled by directed steps 
on a Galton--Watson tree. They investigate events of a player
winning the game in various senses, which have a similar recursive nature as the events
considered here, and they give results about the continuity or discontinuity of phase
transitions for these events \cite[Theorem~5]{HM}.

\subsection{Sketches of proofs}
The proofs of \thref{thm:derivatives,thm:criticality} are fairly straightforward.
For \thref{thm:derivatives}, we express $\Psi(x)$ as a sum of polynomials
and carry out combinatorial calculations to compute their derivatives.
Proving \thref{thm:criticality} is just a matter of Taylor approximation of
$\Psi(x)$ near $x=0$ combined with \thref{prop:interpretable.crit}, our interpretability
criterion from \cite{JPS}. Section~\ref{sec:analytic}
is devoted to these two proofs.

The proof of \thref{thm:critical.event} is more involved.
Given an $m$-supercordant $(\chi,h)$ with smallest nonzero fixed point $x_0$, 
we truncate $\chi$ to form a new child distribution
$\bar\chi$, setting $\bar\chi(\ell)=0$ for $\ell$ in tiers~$m+1$ and above.
From \thref{thm:derivatives}, we know that the system $(\bar\chi,h)$ remains $m$-supercordant.
The hard part of the proof is to show
that $(\bar\chi,h)$ has only a single nonzero fixed point.
We carry this out by decomposing $\bar\chi$ as a mixture of what we call \emph{primitive $m$-critical
measures} that have nice combinatorial properties.
From \thref{lem:highest.fp}, we know that the single nonzero fixed point of $(\bar\chi,h)$ is
associated with the interpretation that $\barT\sim\GW{\bar\chi}$ contains an admissible subtree.
Embedding $\barT$ into $T\sim\GW{\chi}$, we can view this event as $T$ containing
an admissible subtree made up only of vertices with $\ell$ children where $h(\ell)\leq m$.
This event is not an interpretation of $h$---it does not satisfy the recursive property described
by $h$---but it is a subevent of the correct interpretation for $x_0$, which we are able to 
exploit to prove \thref{thm:critical.event}.

\section{Analytic properties of the fixed-point equation}\label{sec:analytic}

We start with the proof of \thref{prop:interpretable.crit}. This proof
belongs more in \cite{JPS} than here, but we give it so that it is spelled out somewhere.
\begin{proof}[Proof of \thref{prop:interpretable.crit}]
  The proof is just a matter of tying together some more general results from
  \cite{JPS}.
  The fixed point $x_0$ has an interpretation
  if and only if the associated pivot tree is subcritical or critical \cite[Theorem~1.7]{JPS}.
  (See \cite{JPS} for the meaning of \emph{pivot tree}.)
  This pivot tree is subcritical or critical if and only if $\Psi'(x_0)\leq 1$ \cite[Lemma~5.3]{JPS}.
\end{proof}

\begin{remark}
  The proof of \cite[Lemma~5.3]{JPS} contains a step of computing $\Psi'(x)$
  by interchanging the order of a derivative and an expectation.
  This step is not justified in the proof, but it is easily shown to hold if
  $\chi$ has finite expectation, one of our assumptions for \thref{prop:interpretable.crit}.
\end{remark}

Now we start our work toward the proofs of \thref{thm:derivatives,thm:criticality}.
We define the polynomials
\begin{align}
  \Be{n}{k} &= \binom{n}{k}x^k(1-x)^{n-k},\label{eq:Be}\\
  \Bg{n}{k} &= \sum_{j=k}^n\Be{n}{k}.\label{eq:Bg}
\end{align}
For $x\in[0,1]$, we have $\Be{n}{k} = \P[\Bin(n,x)=k]$
and $\Bg{n}{k}=\P[\Bin(n,x)\geq k]$.
Note that we take $\binom{0}{0}=1$ and $\binom{n}{k}=0$ if $k>n$
or $k<0$. Using this notation and assuming $h(\ell)\geq 1$, we have 
$\Psi(x) = \sum_{\ell=1}\chi(\ell)\Bg{\ell}{h(\ell)}$.
Thus we can understand the derivatives of $\Psi$ by working out
the derivatives of $\Bg{n}{k}$, which we do now.
\begin{prop}\thlabel{prop:B.deriv}
  For $m\geq 1$,
  \begin{align}\label{eq:B.deriv}
    \frac{d^{m}}{dx^{m}}\Bg{n}{k} &= (n)_m\sum_{j=1}^{m} (-1)^{j+m}\binom{m-1}{j-1}\Be{n-m}{k-j}.
  \end{align}  
\end{prop}
\begin{proof}
  By direct calculation,
  \begin{align}\label{eq:Be.deriv}
    \frac{d}{dx}\Be{n}{k} &= n\bigl( \Be{n-1}{k-1} - \Be{n-1}{k} \bigr).
  \end{align}
  Hence
  \begin{align*}
    \frac{d}{dx}\Bg{n}{k} &= n\sum_{j=k}^n\bigl(\Be{n-1}{j-1} - \Be{n-1}{j}\bigr),
  \end{align*}
  and this sum telescopes to yield $n\Be{n-1}{k-1}$,
  establishing the $m=1$ case.
  
  Now assume the result for $m$ and we prove it for $m+1$. Differentiating the right-hand side
  of \eqref{eq:B.deriv} using \eqref{eq:Be.deriv} gives
  \begin{align*}
    \frac{d^{m+1}}{dx^{m+1}}\Bg{n}{k} &= (n)_m\sum_{j=1}^m(-1)^{j+m}\binom{m-1}{j-1}(n-m)
              \bigl(\Be{n-m-1}{k-j-1} - \Be{n-m-1}{k-j}\bigr)\\
              &= (n)_{m+1}\Biggl(\sum_{j=2}^{m+1}(-1)^{j+m+1}\binom{m-1}{j-2}
              \Be{n-m-1}{k-j} \\&\qquad\qquad\qquad\qquad- \sum_{j=1}^m(-1)^{j+m}\binom{m-1}{j-1}\Be{n-m-1}{k-j}\Biggr)\\
              &= (n)_{m+1}\sum_{j=1}^{m+1}(-1)^{j+m+1}\binom{m}{j-1}\Be{n-m-1}{k-j},
  \end{align*}
  using the identity $\binom{n}{k-1}+\binom{n}{k} = \binom{n+1}{k}$ in the last line.
\end{proof}

\begin{lemma}\thlabel{lem:derivatives.converge}
  Let $\Psi$ be the automaton distribution map for a system $(\chi,h)$.
  Let $\chi_n$ be the truncation of $\chi$ to $n$, i.e., the probability measure satisfying
  $\chi_n(\ell)=\chi(\ell)$ for $\ell\in\{1,\ldots,n\}$ and $\chi_n(0) = \chi\{0,n+1,n+2,\ldots\}$.
  Let $\Psi_n$ be the automaton distribution map for $(\chi_n,h)$.
  If $\chi$ has finite $m$th moment, then $\Psi^{(m)}$ exists and is
  the uniform limit of $\Psi_n^{(m)}$ on $[0,1]$ as $n\to\infty$.
\end{lemma}
\begin{proof}
  We have $\Psi_n(x)=\sum_{\ell=1}^n\chi(\ell)\Bg{\ell}{h(\ell)}$ and
  $\Psi(x)=\lim_{n\to\infty}\Psi_n(x)$. It suffices to show that $\Psi_n^{(k)}$ converges uniformly
  on $[0,1]$
  to some limit as $n\to\infty$ 
  for $0\leq k\leq m$ \cite[Theorem~7.17]{Rudin}. 
  For $k\geq 1$, we apply \eqref{eq:B.deriv} and bound
  $\Be{n}{k}$ by $1$ to obtain
  \begin{align*}
    \abs[\bigg]{\frac{d^k}{dx^k}\Bg{\ell}{h(\ell)}} &\leq (\ell)_k \sum_{j=1}^k\binom{k-1}{j-1}=(\ell)_k2^{k-1}\leq \ell^k2^k.
  \end{align*}
  The same statement for $k=0$, that $\abs{\Bg{\ell}{h(\ell)}}\leq 1$, also holds.
  Hence for all $0\leq k\leq m$,
  \begin{align*}
    \abs[\Bigg]{\sum_{\ell=n+1}^{\infty}\chi(\ell)\frac{d^k}{dx^k}\Bg{\ell}{h(\ell)}}
      &\leq 2^{k}\sum_{\ell=n+1}^{\infty}\chi(\ell)\ell^k,
  \end{align*}
  which vanishes as $n\to\infty$ by our assumption that $\chi$ has finite $m$th moment.
  This demonstrates that $\Psi_n^{(k)}$ converges uniformly as $n\to\infty$, completing the proof.
\end{proof}

\begin{proof}[Proof of \thref{thm:derivatives}]
  Let $\chi_n$ be the truncation of $\chi$ to $n$ and let $\Psi_n$ be the automaton
  distribution map of $(\chi_n,h)$, as in the previous lemma.
  Applying \thref{prop:B.deriv} to each summand of $\Psi_n(x)=\sum_{\ell=1}^n\chi(\ell)\Bg{\ell}{h(\ell)}$ gives
  \begin{align*}
    \Psi_n^{(m)}(x) &= \sum_{\ell=1}^n \chi(\ell)(\ell)_m\sum_{j=1}^m(-1)^{j+m}\binom{m-1}{j-1}\Be{\ell-m}{h(\ell)-j}.
  \end{align*}
  Observing that $\Be[0]{n}{k}=\1\{\text{$n\geq 0$ and $k=0$}\}$, we obtain
  \begin{align*}
    \Psi_n^{(m)}(0) &= \sum_{\ell=1}^n \chi(\ell)(\ell)_m(-1)^{h(\ell)+m}\binom{m-1}{h(\ell)-1}.
  \end{align*}
  Applying \thref{lem:derivatives.converge}, we take $n\to\infty$ to prove
  \eqref{eq:derivatives1}. Equation~\eqref{eq:derivatives2} follows by
  grouping together the terms with $h(\ell)=j$.
\end{proof}

Note that this theorem can fail without the moment assumption:
\begin{example}\thlabel{ex:no.moments}
   Let $\chi(\ell) = 1/\ell(\ell-1)$ for $k\geq 2$, a measure whose expectation is infinite.
   Let $h(\ell)\equiv 2$. Observe that
   \begin{align*}
     \P\bigl[ \Bin(\ell,x) \leq 1 \bigr] = (1-x)^\ell + \ell (1-x)^{\ell-1}x
       &= (1-x)^\ell + \ell (1-x)^{\ell-1}\bigl( 1-(1-x)\bigr)\\
       &= \ell(1-x)^{\ell-1}-(\ell-1)(1-x)^\ell.
   \end{align*}
   Now, let $L\sim\chi$ and compute
   \begin{align*}
     \Psi(x) = \P\bigl[ \Bin(L,x)\geq h(\ell) \bigr] 
     &= 1 - \P\bigl[ \Bin(L,x)\leq 1 \bigr] \\
             &= 1 - \sum_{\ell=2}^{\infty}\chi(\ell)\P\bigl[ \Bin(\ell,x)\leq 1 \bigr]\\
      &= 1 - \sum_{\ell=2}^{\infty}\Bigl( (\ell-1)(1-x)^{\ell-1} - \ell(1-x)^\ell \Bigr).
   \end{align*}
   The sum in the last line telescopes and is equal to $1-x$ for $x\in[0,1]$, yielding $\Psi(x)=x$.
   But this means that $\Psi'(0)=1$ even though \thref{thm:derivatives} would give
   $\Psi'(0)=0$.
\end{example}
We revisit this recursive tree system in \thref{ex:no.moments2} and
show how we arrived at it.

\begin{lemma}\thlabel{lem:continuity}
  For any $m\geq 1$, the function $\chi\mapsto \Psi^{(m)}(0)$ is continuous on $\topol$.
\end{lemma}
\begin{proof}
  Suppose $\chi_n\to\chi$ in $\topol$ and let $\Psi_n$ denote the automaton distribution
  map of $\chi_n$. We need to show that $\Psi_n^{(m)}(0)\to \Psi^{(m)}(0)$. 
  We have
  \begin{align*}
    \abs[\Bigg]{\sum_{\ell\in\tier(j)} \chi_n(\ell)(\ell)_m - \sum_{\ell\in\tier(j)} \chi(\ell)(\ell)_m}
     &\leq \sum_{\ell=1}^{\infty}\abs{\chi_n(\ell)-\chi(\ell)}\ell^m,
  \end{align*}
  which vanishes as $n\to\infty$ by definition of convergence in $\topol$. 
  Hence, by \eqref{eq:derivatives2} from \thref{thm:derivatives}, we have
  $\Psi_n^{(m)}(0)\to \Psi^{(m)}(0)$.
\end{proof}

\begin{proof}[Proof of \thref{thm:criticality}]
  First, suppose that $\Psi'(0)=1$ and $\Psi''(0)< 0$.
  We need to show that for any $\epsilon$, there exist child distributions arbitrarily
  close to $\chi$ in $\topol$ with an interpretable fixed point in $(0,\epsilon)$.
  To show this, we perturb $\chi$ slightly to push its automaton distribution map up,
  creating a new fixed point very close to $0$.
  We accomplish this by transferring some small amount of mass to tier~$1$, which will
  cause $\Psi'(0)$ to increase, as in the phase transition
  shown in Figure~\ref{fig:phase.transition.classes.good}.
  
  We will choose $k$ from tier~$1$ and take $j$ to be either $0$ or some element of a different
  tier, and then transfer mass from $j$ to $k$. First, we must justify that we can find $j$
  and $k$.
  From \eqref{eq:Psi'} and $\Psi'(0)=1$, we know that tier~$1$ is nonempty.
  Choose $k$ arbitrarily from it.
  If $\chi(0)>0$, take $j=0$. If $\chi(0)=0$, then $\chi$ has expectation strictly
  greater than $1$. By \eqref{eq:Psi'}, tier~$1$ does not contain all the mass of $\chi$,
  and therefore some other tier is nonempty; choose $j$ from it.
  
  Now, for $t>0$, define $\chi_t$ by starting with $\chi$ and then
  shifting mass $t$ from $j$ to $k$. Let $\Psi_t$ be the automaton distribution
  map of $\chi_t$. Fix $\epsilon>0$.
  Since $\chi_t\to\chi$ in $\topol$, we just need to show that for sufficiently small $t$,
  the map $\Psi_t$ has an interpretable fixed point in $(0,\epsilon)$.
  By \thref{thm:derivatives},
  \begin{align}
    \Psi'_t(0) &= \Psi'(0)+  kt=1+kt,\label{eq:Psi'1}\\\intertext{and}
    \Psi''_t(0) &< \Psi''(0)<0.\label{eq:Psi'2}
  \end{align}
  By \eqref{eq:Psi'1}, we have $\Psi_t(x)>x$ for sufficiently small $x>0$.
  Choosing $t$ to be small enough relative to $\Psi''(0)$ and 
  applying Taylor approximation, we can force $\Psi_t(x)<x$ for some $x<\epsilon$,
  implying the existence of a fixed point $x_0$ with $\Psi'_t(x_0)<1$, which
  is hence interpretable by \thref{prop:interpretable.crit}.
  
  Now, suppose $\Psi'(0)=1$ and $\Psi''(0)=0$. Again we must show the existence of child distributions
  arbitrarily close to $\chi$ with an interpretable fixed point in $(0,\epsilon)$. We take the same approach as above,
  perturbing $\chi$ to increase $\Psi'(0)$ and decrease $\Psi''(0)$, but we must be careful
  about the rates of increase and decrease.
  Choose $k$ from tier~$1$ as before. From $\Psi''(0)=0$ and \eqref{eq:Psi''}, we know that $\chi$
  assigns positive mass to tier~$2$; choose $j$ from it. Now define
  \begin{align*}
    \chi_t(\ell) = \begin{cases}
      \chi(\ell) + t^2 & \text{if $\ell=k$,}\\
      \chi(\ell) - t & \text{if $\ell=j$,}\\
      \chi(0) + t - t^2 & \text{if $\ell=0$,}\\
      \chi(\ell) & \text{otherwise,}
    \end{cases}
  \end{align*}
  for small values of $t$, and let $\Psi_t$ be the automaton distribution map of $\chi_t$.
  By \eqref{eq:Psi'} and \eqref{eq:Psi''},
  \begin{align*}
    \Psi'_t(0) &= \Psi'(0)+t^2k =  1 + t^2 k,\\
    \Psi''_t(0) &= \Psi''(0)-t j(j-1) -t^2k(k-1)=-t j(j-1) -t^2k(k-1).
  \end{align*}
  Thus $\Psi_t(x)>x$ immediately to the right of $0$,
  and Taylor approximation again shows that when $t$ is sufficently small $\Psi_t(x)<x$
  for some $x<\epsilon$. This proves the existence of a fixed point $x_0\in(0,\epsilon)$
  with $\Psi'_t(x_0)<1$.
  
  Now we consider the converse. Suppose
  $\Psi'(0)\neq 1$. 
  By \thref{lem:continuity},
  we can choose a neighborhood $U\subseteq\topol$ around $\chi$ so that for all $\pi\in U$,
  the automaton distribution map of $(\pi,h)$ has first derivative at $0$ uniformly bounded
  away from $1$ and second derivative at zero uniformly bounded. By Taylor approximation,
  all these maps have no fixed points on $(0,\epsilon)$ for some small $\epsilon>0$, demonstrating
  that $\chi$ is not critical.
  
  Last, suppose $\Psi'(0)=1$ and $\Psi''(0)>0$. By \thref{lem:continuity}, we can choose
  a neighborhood $U\subseteq\topol$ around $\chi$ such that all automaton distribution maps $\Psi_\pi$ of
  $(\pi,h)$ for $\pi\in U$ have second derivative at zero uniformly bounded above $0$
  and third derivative uniformly bounded. Hence for some $\epsilon>0$,
  each map $\Psi_\pi$  is strictly convex on
  $[0,\epsilon]$. The function $\Psi_\pi(x) - x$ is also strictly convex and hence
  has at most two roots on $[0,\epsilon]$. One of them is at $0$. By convexity, any other
  root of $\Psi_\pi(x)-x$ on $[0,\epsilon]$ must occur with the graph crossing the $x$-axis from below
  to above as $x$ increases. Thus $\Psi_\pi$ has derivative greater than $1$ at this fixed point,
  and by \thref{prop:interpretable.crit} it has no interpretation. Since no systems $(\pi,h)$
  for $\pi\in U$ have an interpretable fixed point in $(0,\epsilon)$, the measure $\chi$ is not critical.
\end{proof}

\section{Truncations}

Let the \emph{maximum threshold} of $(\chi, h)$ be the maximum value of $h(\ell)$
over all $\ell$ satisfying $\chi(\ell)>0$.
Define the \emph{$m$-truncation} of $\chi$
as the child distribution $\bar\chi$ where for $\ell\geq 1$,
\begin{align*}
  \bar\chi(\ell) &= 
  \begin{cases}
    \chi(\ell) &  \text{if $h(\ell)\leq m$},\\
    0 & \text{if $h(\ell)>m$},
  \end{cases}
\end{align*}
with $\bar\chi(0)$ set to make $\bar\chi$ a probability measure.
In other words, $\bar\chi$ is obtained from $\chi$ by lopping off tiers $m+1$ and higher,
shifting their weight to $0$.
Recall from \thref{rmk:only.tierm} that
the automaton distribution maps of $(\chi,h)$
and $(\bar\chi,h)$ match to $m$ derivatives at $0$.
Thus $(\bar\chi,h)$ is a system of maximum threshold $m$ or less whose automaton distribution
map approximates that of $(\chi,h)$ near $0$.

The point of this section is to prove the following result, which is a major step in
proving \thref{thm:critical.event}:
\begin{prop}\thlabel{prop:truncations}
  Let $\chi\in\topol$ and let $h(\ell)$ be increasing.
  Suppose that $(\chi, h)$ is $m$-supercordant. Then its $m$-truncation has
  a unique nonzero fixed point.
\end{prop}

The key to this proposition is a thorough understanding of $m$-concordant
recursive tree systems of maximum threshold $m$, which we call \emph{$m$-critical}.
If $(\chi,h)$ has at most one element in each tier, we call it \emph{primitive}.
We will work extensively with primitive $m$-critical recursive tree systems.
An example of such a system is
\begin{align*}
  \chi(\ell)&=\begin{cases}
    2/5 & \text{for $\ell=0$,}\\
    1/2 & \text{for $\ell=2$,}\\
    1/20 & \text{for $\ell=5$},\\
    1/20 & \text{for $\ell=6$,}
  \end{cases}, &
  h(\ell)&=\begin{cases}
      1 & \text{for $\ell=0$,}\\
      1 & \text{for $\ell=2$,}\\
      2 & \text{for $\ell=5$},\\
      3 & \text{for $\ell=6$.}
  \end{cases}
\end{align*}
We can check using \thref{thm:derivatives}
that this system is 3-concordant (i.e., it has $\Psi'(0)=1$ and $\Psi''(0)=\Psi^{(3)}(0)=0$).
It is 3-critical because it is 3-concordant and has maximum threshold $3$,
and it is primitive because tiers~1, 2, and 3 each have one element (recall that the tiers
exclude $0$ by definition).

These recursive tree systems have many good properties. 
In \thref{lem:decompose}, we show that $m$-critical systems
decompose into mixtures (i.e., convex combinations) of primitive $m$-critical systems.
And for the primitive $m$-critical systems $(\chi,h)$, the automaton
distribution map has a useful connection with a martingale $(X_n)_{n\geq 1}$
we describe now.

First, we define a time-inhomogenous Markov chain $(R_n)_{n\geq 1}$ as follows.
Let $R_1=1$. Then, conditional on $R_n$, let 
\begin{align}
  R_{n+1} &=\begin{cases}
    R_n & \text{with probability $\frac{n+1-R_n}{n+1}$,}\\
    R_n+1 & \text{with probability $\frac{R_n}{n+1}$.}
  \end{cases}\label{eq:Rn}
\end{align}
There is an alternative construction of $(R_n)_{n\geq 1}$ that yields some insight.
Start with the permutation $\sigma_1$ of length~$1$.
At each step, form $\sigma_{n+1}$ from $\sigma_n$ by viewing $\sigma_n$
in one-line notation and inserting the digit $n+1$ uniformly at random into the $n+1$ possible locations.
For example, if $\sigma_3=213$, then $\sigma_4$ is equally likely to be each of $4213$,
$2413$, $2143$, and $2134$. Then let $R_n=\sigma_n^{-1}(1)$, the location of $1$
in the one-line notation of $\sigma_n$. 
When we insert $n+1$ into $\sigma_n$ to form $\sigma_{n+1}$, it has 
probability~$(n+1-R_n)/(n+1)$ of landing to the right of $1$ and
probability~$R_n/(n+1)$ of landing to the left of $1$, matching the dynamics
given in \eqref{eq:Rn}. We note one consequence of this perspective:
\begin{lemma}\thlabel{lem:uniform}
  The random variable $R_n$ is uniformly distributed over $\{1,\ldots,n\}$.
\end{lemma}
\begin{proof}
  First, we argue by induction that $\sigma_n$ is a uniformly random permutation of length $n$,
  with $n=1$ as the trivial base case. To extend the induction,
  let $\tau$ be an arbitrary permutation of length~$n$ and let $\tau'$ be the permutation of
  of length~$n-1$ obtained by deleting $n$ from the one-line notation form of $\tau$.
  Then $\sigma_n$ can be equal to $\tau$ only if $\sigma_{n-1}=\tau'$, and we compute
  \begin{align*}
    \P[\sigma_n=\tau] = \P[\sigma_{n-1}=\tau']\,\P[\sigma_n=\tau\mid\sigma_{n-1}=\tau'] = \frac{1}{(n-1)!}\cdot\frac{1}{n}=\frac{1}{n!}
  \end{align*}
  by the inductive hypothesis and definition of $\sigma_n$.
  
  To complete the proof, observe that $R_n\eqd\sigma_n^{-1}(1)$ and is hence
  uniform over $\{1,\ldots,n\}$.
\end{proof}

Finally, we give the sequence $(X_n)_{n\geq 1}$ and show that it
is a martingale. We define it in terms of $R_n$ and the polynomials
$\Bg{n}{k}$ defined in \eqref{eq:Bg}.

\begin{lemma}\thlabel{lem:mtg}
  Fix $x\in[0,1]$ and define
  \begin{align*}
    X_n&=\Bg{n}{R_n}.
  \end{align*}
  Then $(X_n)_{n\geq 1}$ is a martingale
  adapted to the filtration $\Ff_n=\sigma(R_1,\ldots,R_n)$.
\end{lemma}
\begin{proof}
  First, we claim that
  \begin{align}\label{eq:bin.identity}
    \Bg{n}{k} = \Bg{n+1}{k+1} + \frac{n+1-k}{n+1}\Be{n+1}{k}.
  \end{align}
  To see this, consider $n+1$ independent trials with success probability $x$.
  Then eliminate one at random and consider the event that there are at least
  $k$ successes in the remaining $n$ trials. 
  The left-hand side of \eqref{eq:bin.identity} is the probability of this event, which
  occurs if either (a) there were $k+1$ or more successes in the original set
  of trials, or (b) there were exactly $k$ successes but the trial removed was a failure.
  Then (a) occurs with probability $\Bg{n+1}{k+1}$ and (b) with probability $\frac{n+1-k}{n+1}\Be{n+1}{k}$,
  proving \eqref{eq:bin.identity}.
  
  Now, we compute
  \begin{align*}
    \E [X_{n+1}\mid \Ff_n] &= \E\Biggl[ \sum_{k=R_{n+1}}^{n+1} \Be{n+1}{k} \Biggmid \Ff_n \Biggr]\\
      &= \frac{n+1-R_n}{n+1} \sum_{k=R_n}^{n+1}\Be{n+1}{k} 
        + \frac{R_n}{n+1} \sum_{k=R_n+1}^{n+1}\Be{n+1}{k}\\
      &=  \Bg{n+1}{R_n+1} + \biggl(\frac{n+1-R_n}{n+1}\biggr) \Be{n+1}{R_n} = \Bg{n}{R_n},
  \end{align*}
  applying \eqref{eq:bin.identity} in the last line.
  Hence $\E[X_{n+1}\mid\Ff_n] = X_n$, confirming that the sequence is a martingale.
\end{proof}

We will make use of this martingale by applying the optional stopping theorem
to assert that $x=X_1=\E X_T$ for various stopping times $T$. This expectation has the form
\begin{align*}
  \E X_T = \sum_{n=1}^{\infty}\P[T=n]\,\P[\Bin(n,x)\geq R_n\mid T=n].
\end{align*}
For a recursive tree system $(\chi,h)$ where $\chi(n)=\P[T=n]$, if $T$ is chosen so that
$R_n=h(n)$ when $T=n$, this expression is nearly the same as $\Psi(x)$.
The following example is off track for the section, but it illustrates how to use this idea.
\begin{example}[\thref{ex:no.moments} revisited]\thlabel{ex:no.moments2}
  In \thref{ex:no.moments}, we showed that the system $(\chi,h)$ with $\chi(n)=1/n(n-1)$
  for $n\geq 2$ and $h(n)\equiv 2$ has automaton distribution map $\Psi(x)=x$ by an explicit
  calculation. Now we give a new proof that demonstrates how we arrived at the example.
  Let $T$ be the first time the chain $R_n$ jumps from $1$ to $2$; that is,
  $T=\min\{n\colon R_n=2\}$. By the chain's dynamics, for $n\geq 2$
  \begin{align*}
    \P[T\geq n] &= \P[R_{n-1}=1] = \prod_{k=2}^{n-1}\frac{k-1}{k}= \frac{1}{n-1},\\\intertext{and}
    \P[T=n\mid T\geq n] &= \P[R_n=2\mid R_{n-1}=1] = \frac1n.
  \end{align*}
  Putting these together, we have $\P[T=n] = 1/n(n-1)$ for $n\geq 2$. 
  Also observe that $T<\infty$ with probability~1,
  since $\P[T\geq n]\to 0$ as $n\to\infty$.
  
  Now, we set $\chi(n)=\P[T=n]$ and $h(n)\equiv 2$. Since $\abs{X_n}\leq 1$, the optional stopping
  theorem applies and yields
  \begin{align*}
    x = \E X_T = \sum_{n=2}^{\infty}\P[T=n]\,\P\bigl[\Bin(n,x)\geq R_n\mid T=n\bigr] = \sum_{n=2}^{\infty}\chi(n)\P\bigl[ \Bin(n,x)\geq 2 \bigr] = \Psi(x).
  \end{align*}
\end{example}

The next result uses the optional stopping theorem in the same way to
compute primitive $m$-critical systems with a given set of support. Recall that
\emph{primitive} means that each tier
has at most one element, i.e., the values $h(\ell)$ are distinct
for all $\ell\geq 1$ in the support of the child distribution.

\begin{lemma}\thlabel{lem:crit.seq}
  For any sequence of integers $1\leq\ell_1<\cdots<\ell_m$, there is a unique probability
  measure $\chi$ supported within $\{0,\ell_1,\ldots,\ell_m\}$ 
  so that the system $(\chi,h)$ with $h(\ell_k)=k$ is $m$-concordant.
  We denote this measure $\chi$ by the notation $\crit(\ell_1,\ldots,\ell_m)$.
  For this system $(\chi,h)$ with automaton distribution map $\Psi$:
  \begin{enumerate}[(a)]
    \item $\chi$ satisfies
        \begin{align*}
          \chi(\ell_1) &= \frac{1}{\ell_1},\\
          \chi(\ell_k) &=
          \frac{1}{(\ell_k)_k}\sum_{j=1}^{k-1}(-1)^{k+j+1}\binom{k-1}{j-1}\chi(\ell_j)(\ell_j)_k,
          \qquad \text{for $2\leq k\leq m$.}
        \end{align*} \label{i:crit.seq.combint}
  \end{enumerate}
  If $\ell_1=1$, then $\chi=\delta_1$ and $\Psi(x)=x$.
  If $\ell_1\geq 2$, then the following properties hold as well:
  \begin{enumerate}[(a), resume]
    \item $\chi$ places positive mass on each of $\{0,\ell_1,\ldots,\ell_m\}$;
        in particular, $(\chi,h)$ has maximum threshold $m$
        and is therefore $m$-critical; \label{i:crit.seq.fullsupport}
    \item $\bigl(x-\Psi(x)\bigr)/\chi(0)$
      is a convex combination of the polynomials
      \begin{align*}
        \Bg{\ell_m}{r},\quad r\in\{m+1,\ldots,\ell_m\};
      \end{align*}\label{i:crit.seq.b}
    \item $(\chi,h)$ is $(m+1)$-subcordant. \label{i:crit.seq.subc}
  \end{enumerate}  
\end{lemma}
\begin{proof}
  Fix a sequence $1\leq\ell_1<\cdots<\ell_m$ and let $h(\ell_k)=k$.
  We will first prove that if there exists $m$-concordant $\chi$ supported on 
  $\{0,\ell_1,\ldots,\ell_m\}$, then \ref{i:crit.seq.combint}
  holds. This shows that such a $\chi$ is unique, if it exists, since we can
  apply \ref{i:crit.seq.combint} inductively to determine $\chi(\ell_1),\ldots,\chi(\ell_m)$.
  (Note that it is not obvious a priori that the values of $\chi(\ell_k)$ given by these formulas
  are positive numbers or that their sum is $1$ or smaller,
  which is why we cannot construct $\chi$ by this formula.)
  After this, we will prove existence
  of $\chi$, and last we show \ref{i:crit.seq.fullsupport}--\ref{i:crit.seq.subc}.
  
  To prove \ref{i:crit.seq.combint} under the assumption of existence of $\chi$, 
  we simply apply \thref{thm:derivatives}. In the $k=1$ case we use
  \eqref{eq:Psi'}, yielding $1=\Psi'(0)=\chi(\ell_1)\ell_1$ and proving
  that $\chi(\ell_1)=1/\ell_1$.
  Similarly, for $2\leq k\leq m$, we have $0=\Psi^{(k)}(0)$ and we apply \eqref{eq:derivatives2}
  to deduce the rest of \ref{i:crit.seq.combint}.

  Now, we show that $\chi$ exists.
  Consider the chain $(R_n)_{n\geq 1}$ defined previously, and let
  \begin{align*}
    T = \min\{\ell_k\colon R_{\ell_k}=k\},
  \end{align*}
  with $T=\infty$ if $R_{\ell_k}\neq k$ for $k=1,\ldots,n$. The random variable $T$
  is a stopping time for the filtration $(\Ff_n)_{n\geq 1}$ defined in \thref{lem:mtg}.
  To get a feeling for $T$, consider the perspective of $R_n$ as the location
  of $1$ in the one-line notation of a growing random permutation $\sigma_n$, as 
  described before \thref{lem:uniform}. The idea for $T$ is that we only consider stopping
  at times $\ell_1,\ell_2,\ldots$, and that we stop at the first time $\ell_k$ where
  $1$ is in position~$k$ in $\sigma_{\ell_k}$.
  
  We define $\chi$ be setting $\chi(\ell_k) = \P[T=\ell_k]$ for $k=1,\ldots,m$ and setting
  $\chi(0)=\P[T=\infty]$. Clearly this is a probability measure supported
  within $\{0,\ell_1,\ldots,\ell_m\}$. 
  We now compute
  \begin{align*}
    \Psi(x) &= \sum_{\ell=1}^\infty \chi(\ell)\Bg{\ell}{h(\ell)} = \sum_{k=1}^m \P[T=\ell_k]\Bg{\ell_k}{k}.
  \end{align*}
  As in \thref{ex:no.moments2}, this closely resembles $\E X_T$ for the martingale $(X_n)_{n\geq 1}$
  defined in \thref{lem:mtg},
  since $\Bg{\ell_k}{k}=X_T$ when $T=\ell_k$. Indeed, by the optional stopping theorem,
  \begin{align*}
    x = X_1 = \E X_{T\wedge \ell_m} &= \sum_{k=1}^m\P[T=\ell_k]\E[X_T\mid T=\ell_k] + \P[T=\infty]\E[X_{\ell_m}\mid T=\infty]\\
    &= \Psi(x) +\chi(0)\sum_{r=1}^{\ell_m}\P[R_{\ell_m}=r\mid T=\infty]\Bg{\ell_m}{r}.
  \end{align*}
  We claim that if $T=\infty$, then $R_{\ell_m}\geq m+1$.
  Indeed, if $T>\ell_1$, then $R_{\ell_1}\neq 1$, and hence $R_{\ell_1}\geq 2$.
  Since $(R_n)_{n\geq 1}$ is increasing, we then have $R_{\ell_2}\geq 2$.
  If $T>\ell_2$, then $R_{\ell_2}\neq 2$, and hence $R_{\ell_2}\geq 3$.
  Continuing in this way, if $T>\ell_j$
  then $R_{\ell_j}\geq j+1$. In particular, if $T=\infty$, then $R_{\ell_m}\geq m+1$, proving the claim. 
  Hence $\P[R_{\ell_m}=r\mid T=\infty]=0$ for $r\leq m$, yielding
  \begin{align}
    \Psi(x) = x - \chi(0)\sum_{r=m+1}^{\ell_m}\P[R_{\ell_m}=r\mid T=\infty]\Bg{\ell_m}{r}.\label{eq:Psi.poly}
  \end{align}

  Now we can confirm that the system $(\chi,h)$ we have constructed is $m$-concordant.
  Since the polynomials $\Bg{\ell_m}{r}$ for $r\geq m+1$
  are divisible by $x^{m+1}$, we rewrite \eqref{eq:Psi.poly} as
  \begin{align*}
    \Psi(x) = x - \chi(0)x^{m+1} F(x),
  \end{align*}
  where $F(x)$ is a polynomial. Since
  \begin{align*}
    \frac{d^k}{dx^k}\biggr|_{x=0} \bigl(x^{m+1}F(x)\bigr) = 0
  \end{align*}
  for $k=1,\ldots,m$, we see that $\Psi'(0) = 1$ and
  $\Psi^{(k)}(0) = 0$ for $2\leq k\leq m$. Thus we have shown existence of $\chi$ so that
  $(\chi,h)$ is an $m$-critical
  system with support $\{0,\ell_1,\ldots,\ell_m\}$. We have already shown that
  that there can be at most one probability measure with this property.
  Henceforth we denote the measure $\chi$ we have constructed by $\crit(\ell_1,\ldots,\ell_m)$.
  
  When $\ell_1=1$, we have $T=1$ a.s., which makes $\chi=\delta_1$
  and $\Psi(x)=x$. From now on  
  we assume $\ell_1\geq 2$.
  To prove \ref{i:crit.seq.fullsupport}, we must show $\crit(\ell_1,\ldots,\ell_m)$
  assigns strictly positive mass to each of $0,\,\ell_1\ldots,\,\ell_m$.
  We just need to show that the events $\{T=\ell_k\}$ for $k=1,\ldots,m$
  and the event $\{T=\infty\}$ have positive probability. We claim that
  $T=\ell_k$ occurs if $R_1,\,R_2,\,\ldots,\,R_{\ell_k}$ is the sequence
  \begin{align*}
    1,\,2,\ldots,\,k,\,k,\ldots,\,k.
  \end{align*}
  Indeed, for $j<k$ either $R_{\ell_j}= \ell_j$ or $R_{\ell_j}=k$.
  Since $\ell_1>1$, we have $\ell_j>j$, and hence we stop only when we reach $\ell_k$,
  proving the claim. Similarly,  
  and $T=\infty$ occurs if $R_1,\ldots,R_{\ell_m}$ is
  \begin{align*}
    1,\,2,\ldots,\, m+1,\,m+1,\ldots,\,m+1.
  \end{align*}
  By the dynamics of the chain $(R_n)$, it has positive probability of taking on
  these sequences.

  Property~\ref{i:crit.seq.b} follows directly from \eqref{eq:Psi.poly} together with
  $\chi(0)>0$ from \ref{i:crit.seq.fullsupport}.
  It remains to prove \ref{i:crit.seq.subc} by showing that $\Psi^{(m+1)}(0)<0$.
  From \ref{i:crit.seq.b} we have $\Psi(x)\leq x$.
  If $\Psi^{(m+1)}(0)>0$, then by Taylor approximation we would have $\Psi(x)> x$ for sufficiently
  small $x$, a contradiction. Hence $\Psi^{(m+1)}(0)\leq 0$. To rule out $\Psi^{(m+1)}(0)=0$,
  we make use of uniqueness.
  Choose any $\ell_{m+1}>\ell_m$ and
  extend $h$ by setting $h(\ell_{m+1})=m+1$. 
  If $\Psi^{(m+1)}(0)=0$, then $(\chi,h)$ is
  $(m+1)$-concordant and supported within $\{0,\ell_1,\ldots,\ell_{m+1}\}$.
  But by what we have already proven, the unique measure with these properties
  places positive weight on $\ell_{m+1}$, a contradiction since $\chi(\ell_{m+1})=0$.
  Hence $\Psi^{(m+1)}(0)<0$, completing the proof.
\end{proof}

\begin{remark}
  \thref{lem:crit.seq} has a combinatorial interpretation.
  Let $\sss_n$ denote the set of permutations of $\{1,\ldots,n\}$.
  Fix $1<\ell_1<\cdots<\ell_m$, and for $\pi\in\sss_m$ consider
  the sequence of permutations $\pi_1,\ldots,\pi_m=\pi$
  where $\pi_k\in\sss_{\ell_k}$ is obtained from $\pi$ by deleting values larger than $\ell_k$
  from the one-line notation of $\pi$. For example, if $(\ell_1,\ell_2,\ell_3)=(3,4,6)$ and
  $\pi=621435$, then 
  \begin{align*}
    (\pi_1,\pi_2,\pi_3)=(213,\,2143,\,621435).
  \end{align*}
  Now, let $A_k$ consist of all $\pi\in\sss_m$ such that $\pi_k^{-1}(1)=k$
  but $\pi_j^{-1}(1)\neq j$ for $j<k$. In other words, $A_k$ is made up of
  the permutations $\pi$ in which $1$ is in position $j$ in $\pi_j$ for the first time
  when $j=k$. For example, in the example above, $621435\in A_2$, because $1$ is not in position~$1$
  in $\pi_1$ but is in position~$2$ in $\pi_2$.
  
  Counting the number of permutations in $A_k$ corresponds to computing $\P[T=k]$
  in the proof of \thref{lem:crit.seq}. Thus, for $\chi=\crit(\ell_1,\ldots,\ell_m)$,
  we have $\abs{A_k}=\ell_m!\chi(k)$.
  \thref{lem:crit.seq}\ref{i:crit.seq.combint} then yields:
  \begin{align*}
    \abs{A_1} &= \frac{\ell_m!}{\ell_1},\\
    \abs{A_k} &= 
    \frac{1}{(\ell_k)_k}\sum_{j=1}^{k-1}(-1)^{k+j+1}\binom{k-1}{j-1}\abs{A_j}(\ell_j)_k,
          \qquad \text{for $2\leq k\leq m$.}
  \end{align*}
  In this form, the formula suggests an explanation via the inclusion-exclusion principle, though
  we could not come up with one.
\end{remark}

\thref{lem:crit.seq} gives us an excellent understanding of primitive $m$-critical
systems. We will extend our knowledge to the nonprimitive $m$-critical systems
in \thref{lem:decompose}. First we need a technical lemma.

\begin{lemma}\thlabel{lem:decompose.multiples}
  Let $(\chi,h)$ and $(\bar\chi,h)$ both be $m$-concordant with maximum threshold $m$ or less.
  Assume that $\chi$ has finite support, that $h(\ell)$ is increasing,
  that $h(\ell)\leq\ell$ for $\ell\geq 1$, and that $\tier_{\chi,h}(k)$ is nonempty
  for each $k\in\{1,\ldots,m\}$.
  Suppose that for all $k\in\{1,\ldots,m\}$, the vector
  $(\bar\chi(p))_{p\in\tier(k)}$ is a scalar multiple of $(\chi(p))_{p\in\tier(k)}$.
  Then $\chi=\bar\chi$.
\end{lemma}
\begin{proof}
  Let $\Psi$ and $\bar\Psi$ be the automaton distribution maps
  of $(\chi,h)$ and $(\bar\chi,h)$, respectively.
  Let $\alpha_k$ be the scalar satisfying $(\bar\chi(p))_{p\in\tier(k)}=\alpha_k(\chi(p))_{p\in\tier(k)}$.
  We will show that $\alpha_k=1$ for all $k\in\{1,\ldots,m\}$ by induction on $k$. The idea is that
  $m$-concordancy together with $\alpha_1=\cdots=\alpha_{k-1}=1$ together with \thref{thm:derivatives}
  implies that $\alpha_k=1$.
  
  For the $k=1$ case, we see from \eqref{eq:Psi'}
  that $\bar\Psi'(0)=\alpha_1\Psi'(0)$. Since $\chi$ and $\bar\chi$ are $m$-concordant,
  we have $\Psi'(0)=\bar\Psi'(0)=1$, showing that $\alpha_1=1$.   
  Next, assume $\alpha_1=\cdots=\alpha_{k-1}=1$, and we show that $\alpha_{k}=1$.
  By $m$-concordancy of $\chi$ and $\bar\chi$, the inductive hypothesis, and \eqref{eq:derivatives2},
  \begin{align*}
    0 &= \Psi^{(k)}(0) =  \sum_{\ell\in\tier(k)}\chi(\ell)(\ell)_k + \sum_{j=1}^{k-1}(-1)^{k+j}\binom{k-1}{j-1}\sum_{\ell\in\tier(j)}\chi(\ell)(\ell)_k,\\\intertext{and}
        0 &= \bar\Psi^{(k)}(0) =  \alpha_k\sum_{\ell\in\tier(k)}\chi(\ell)(\ell)_k + \sum_{j=1}^{k-1}(-1)^{k+j}\binom{k-1}{j-1}\sum_{\ell\in\tier(j)}\chi(\ell)(\ell)_k.
  \end{align*}
  By our assumption that $\tier(k)$ is nonempty and that $h(\ell)\leq \ell$, the term
  $\sum_{\ell\in\tier(k)}\chi(\ell)(\ell)_k$ is nonzero. It follows that $\alpha_k=1$.
\end{proof}

\begin{prop}\thlabel{lem:decompose}
  Let $(\chi,h)$ be $m$-concordant with maximum threshold $m$ or less.
  Assume that $\chi$ has finite support, that $h(\ell)$ is increasing, and
  that $h(\ell)\leq\ell$ for $\ell\geq 1$.
  Then
  
  \begin{enumerate}[(a)]
    \item $\chi$ is a convex combination of measures $\crit(\ell_1,\ldots,\ell_m)$ where
      each $\ell_k$ is in tier~$k$ of $(\chi,h)$;\label{i:decompose}
    \item if $\chi\neq\delta_1$, then $(\chi,h)$ has maximum threshold exactly $m$
      and is hence $m$-critical.\label{i:fulltiers}
  \end{enumerate}
\end{prop}

\begin{proof}
  For each $k\in\{1,\ldots,m\}$ and $\ell\in\tier(k)$, define
  \begin{align}\label{eq:a.ell}
    a_\ell = \frac{\chi(\ell)(\ell)_k}{\sum_{i\in\tier(k)}\chi(i)(i)_k}.
  \end{align}
  Note that the denominator in this expression is nonzero: the sum includes 
  $\chi(\ell)(\ell)_k$, and $\chi(\ell)>0$ for $\ell\in\tier(k)$ by definition of $\tier(k)$,
  and $\ell\geq k$ by our assumption that $h(\ell)\leq\ell$.
  Now, we define
  \begin{align}
    \widetilde\chi=\sum_{\ell_1\in\tier(1)}\cdots\sum_{\ell_m\in\tier(m)}a_{\ell_1}\cdots a_{\ell_m}\crit(\ell_1,\ldots,\ell_m).\label{eq:decompose}
  \end{align}
  If tiers $1,\ldots,m$ are nonempty, then this sum
  is a convex combination, since $\sum_{\ell_k\in\tier(k)} a_{\ell_k}=1$ and hence
  \begin{align*}
    \sum_{\ell_1\in\tier(1)}\cdots\sum_{\ell_m\in\tier(m)}a_{\ell_1}\cdots a_{\ell_m}
      &=\Biggl( \sum_{\ell_1\in\tier(1)} a_{\ell_1}\Biggr)\cdots\Biggl( \sum_{\ell_m\in\tier(m)} a_{\ell_m}\Biggr)=1.
  \end{align*}
  As a convex combination of $m$-concordant measures, $\widetilde\chi$ itself is $m$-concordant.
  Indeed,
  its automaton distribution map $\widetilde\Psi$ satisfies
  $\widetilde\Psi'(0)=1$ and $\widetilde\Psi^{(k)}(0)=0$ for $2\leq k\leq m$ since it is
  then a convex combination of functions satisfying the same derivative condition.
  In fact, we will eventually show that $\chi=\widetilde\chi$ and that the assumption of
  nonempty tiers is automatically satisfied under the assumptions of this proposition.
  See \thref{ex:decomposition} to see this decomposition in practice.

  We first argue that $\chi=\widetilde\chi$ under the assumption that tiers 
  $1,\ldots, m$ of $(\chi,h)$ are nonempty.
  With the aim of applying \thref{lem:decompose.multiples}, we will show that
    the vector $(\widetilde\chi(p))_{p\in\tier(k)}$ is a scalar multiple of $(\chi(p))_{p\in\tier(k)}$ for each 
    $k=1,\ldots,m$.
  To prove this, we define
  \begin{align*}
    b_1 &= \frac{1}{\sum_{i\in\tier(1)}\chi(i)i},\\\intertext{and}
    b_k(x_1,\ldots,x_{k-1})
      &=\frac{1}{\sum_{i\in\tier(k)}\chi(i)(i)_k}\sum_{j=1}^{k-1}(-1)^{k+j+1}\binom{k-1}{j-1}\chi(x_j)(x_j)_k,
  \end{align*}
  for $k\geq 2$.
  The denominators on the right-hand side of these equations are nonzero by our assumption 
  of nonempty tiers and that $h(\ell)\leq \ell$.
  Now, fix an arbitrary $k\in\{1,\ldots,m\}$.
  For any $\ell_k\in\tier(k)$, we have 
  \begin{align*}
     a_{\ell_k}\crit(\ell_1,\ldots,\ell_m)\{\ell_k\} &= \chi(\ell_k)b_k(\ell_1,\ldots,\ell_{k-1})
  \end{align*}
  by \thref{lem:crit.seq}\ref{i:crit.seq.combint},
  using the notation $\crit(\ell_1,\ldots,\ell_m)\{i\}$ to denote
  the mass placed on the value $i$ by the measure $\crit(\ell_1,\ldots,\ell_m)$.
  Now, let $p\in\tier(k)$. Since $\crit(\ell_1,\ldots,\ell_m)$ is supported on $\{\ell_1,\ldots,\ell_m\}$,
  the following holds for any $\ell_1,\ldots,\ell_{k-1},\ell_{k+1},\ldots,\ell_m$ with $\ell_i\in\tier(i)$:
  \begin{align*}
    \sum_{\ell_k\in\tier(k)} a_{\ell_1}\cdots &a_{\ell_m}\crit(\ell_1,\ldots,\ell_m)\{p\}\\
     &=
      a_{\ell_1}\cdots a_{\ell_{k-1}}a_p a_{\ell_{k+1}}\cdots a_{\ell_m}\crit(\ell_1,\ldots,\ell_{k-1},p,\ell_{k+1},\ldots,\ell_m)\{p\}\\
    &= a_{\ell_{k-1}}a_{\ell_{k+1}}\cdots a_{\ell_m} \chi(p)b_k(\ell_1,\ldots,\ell_{k-1}).
  \end{align*}
  Now we use this to compute
  \begin{align*}
    \widetilde\chi(p) &= \sum_{\ell_1\in\tier(1)}\cdots\sum_{\ell_m\in\tier(m)}a_{\ell_1}\cdots a_{\ell_m}\crit(\ell_1,\ldots,\ell_m)\{p\}\\
      &= \chi(p)\overbrace{\sum_{\ell_1\in\tier(1)}\cdots\sum_{\ell_m\in\tier(m)}}^{\text{$k$th sum omitted}}
      a_{\ell_1}\cdots a_{\ell_{k-1}}a_{\ell_{k+1}}\cdots a_{\ell_m}b_k(\ell_1,\ldots,\ell_{k-1}).
  \end{align*}
  Thus, for each $p\in\tier(k)$, we have shown that $\widetilde\chi(p)$ is equal to $\chi(p)$ scaled by
  a factor not depending on $p$.
  This completes the proof that
  $(\widetilde\chi(p))_{p\in\tier(k)}$ is a scalar multiple of $(\chi(p))_{p\in\tier(k)}$ for each 
    $k=1,\ldots,m$.
  
  As we noted earlier, $\widetilde\chi$ is $m$-concordant. \thref{lem:decompose.multiples}
  applies and shows that $\chi=\widetilde\chi$. Thus part~\ref{i:decompose} of
  the proposition is proven under the
  extra assumption of nonempty tiers.
  
  Finally, we show that this assumption
  holds whenever $\chi\neq\delta_1$. This proves \ref{i:fulltiers}, and it proves \ref{i:decompose}
  when $\chi\neq\delta_1$. This will complete the proof, since \ref{i:decompose} is trivial when 
  $\chi=\delta_1$ since $\chi=\crit(1)$.
  Thus, we suppose that
  $(\chi,h)$ satisfies all the conditions of the proposition and has an empty tier.
  Let $k$ be the smallest value so that
  tier~$k$ is empty. From \eqref{eq:Psi'} and $\Psi'(0)=1$, we have $k\geq 2$.
  Let $\bar\chi$ be the $k$-truncation of $\chi$, which is equal to the $(k-1)$-truncation
  since tier~$k$ is empty.
  Now $(\bar\chi,h)$ is $(k-1)$-concordant with maximum threshold $k-1$, 
  and $h(\ell)$ is still increasing, and $\bar\chi$ has finite
  support. Also tiers~$1,\ldots,\,k-1$ are all nonempty in $\bar\chi$. 
  Thus all conditions of this proposition are satisfied with $m$ as $k-1$
  as well as the nonempty tiers assumption,
  and therefore $\bar\chi$ decomposes into a convex combination of measures
  $\crit(\ell_1,\ldots,\ell_{k-1})$ with $\ell_i\in\tier(i)$. 
  
  By definition of the $k$-truncation, the measures $\chi$ and $\bar\chi$
  place the same weight on all values in tiers~$1,\ldots,k$. By 
  \thref{rmk:only.tierm},
  this implies that $\bar\Psi^{(k)}(0)=\Psi^{(k)}(0)$. And since $\chi$ is
  $m$-concordant, we have $\Psi^{(k)}(0)=0$ and can conclude that $\bar\Psi^{(k)}(0)=0$
  as well.
  
  Now, consider the decomposition of $\bar\chi$  into a convex combination of measures
  of the form $\crit(\ell_1,\ldots,\ell_{k-1})$. Let $\Psi_{\ell_1,\ldots,\ell_{k-1}}$ denote
  the automaton distribution maps of these measures, and note that
  $\bar\Psi$ is a convex combination of these maps.
  By \thref{lem:crit.seq}\ref{i:crit.seq.subc}, each measure $\crit(\ell_1,\ldots,\ell_{k-1})$
  is either $k$-subcordant or is equal to $\delta_1$. In the first case, 
  $\Psi_{\ell_1,\ldots,\ell_{k-1}}^{(k)}(0)<0$, and in the second
  $\Psi_{\ell_1,\ldots,\ell_{k-1}}^{(k)}(0)=0$. Since $\bar\Psi^{(k)}(0)=0$,
  we have $\crit(\ell_1,\ldots,\ell_{k-1})=\delta_1$ for all measures in the decomposition
  of $\bar\chi$, and therefore $\bar\chi=\delta_1$. But if the $k$-truncation of $\chi$
  is $\delta_1$, then $\chi$ itself is equal to $\delta_1$. Hence $\chi=\delta_1$ if
  any of tiers $1,\ldots,m$ are empty.
\end{proof}

Here is an example of the decomposition of a critical system
into primitive ones. We find it helpful for understanding both \thref{lem:crit.seq,lem:decompose}.

\begin{example}\thlabel{ex:decomposition}
  Let
  \begin{align*}
  \chi(\ell)=\begin{cases}
    64/135 & \text{for $\ell=0$,}\\
    1/3 & \text{for $\ell=2$,}\\
    1/9 & \text{for $\ell=3$},\\
    1/27 & \text{for $\ell=4$,}\\
    2/45 & \text{for $\ell=5$,}\\
  \end{cases}\qquad\qquad \text{and} \qquad\qquad
  h(\ell)=\begin{cases}
      1 & \text{for $\ell=0$,}\\
      1 & \text{for $\ell=2$,}\\
      1 & \text{for $\ell=3$,}\\
      2 & \text{for $\ell=4$},\\
      2 & \text{for $\ell=5$.}
  \end{cases}
\end{align*}
We claim that $(\chi,h)$ is $2$-critical. Indeed, it has maximum threshold $2$, and by \eqref{eq:Psi'}
and \eqref{eq:Psi''}, the automaton distribution map $\Psi$ satisfies
\begin{align*}
  \Psi'(0) &= \chi(2)(2) + \chi(3)(3) = 1,\\
  \Psi''(0) &= \chi(4)(4)(3) + \chi(5)(5)(4) - \chi(2)(2)(1) - \chi(3)(3)(2) = 0.
\end{align*}
By \thref{lem:decompose}, we can express $\chi$ as a mixture of the measures $\crit(2,4)$,
$\crit(2,5)$, $\crit(3,4)$, and $\crit(3,5)$.
First, we compute these measures, which can be done most easily using \thref{lem:crit.seq}\ref{i:crit.seq.combint}:
\begin{align*}
  \crit(2,4)\{\ell\} &= \begin{cases}
     5/12 & \text{for $\ell=0$,}\\
     1/2 & \text{for $\ell=2$,}\\
     1/12 & \text{for $\ell=4$,}
  \end{cases}
  &
  \crit(3,4)\{\ell\} &= \begin{cases}
     1/2 & \text{for $\ell=0$,}\\
     1/3 & \text{for $\ell=3$,}\\
     1/6 & \text{for $\ell=4$,}
  \end{cases}\\[4pt]
  \crit(2,5)\{\ell\} &= \begin{cases}
     9/20 & \text{for $\ell=0$,}\\
     1/2 & \text{for $\ell=2$,}\\
     1/20 & \text{for $\ell=5$,}
  \end{cases}
  &
  \crit(3,5)\{\ell\} &= \begin{cases}
     17/30 & \text{for $\ell=0$,}\\
     1/3 & \text{for $\ell=3$,}\\
     1/10 & \text{for $\ell=5$.}
  \end{cases}
\end{align*}
We can find a decomposition of $\chi$
using the technique from the
proof of \thref{lem:decompose}. We apply \eqref{eq:a.ell} to compute
\begin{align*}
  a_2 &= \frac{\chi(2)(2)}{\chi(2)(2) + \chi(3)(3)} = \frac23,& a_4 &= \frac{\chi(4)(4)(3)}{\chi(4)(4)(3) + \chi(5)(5)(4)}=\frac13,\\
  a_3 &= \frac{\chi(3)(3)}{\chi(2)(2) + \chi(3)(3)} = \frac13, & a_5 &=\frac{\chi(5)(5)(4)}{\chi(4)(4)(3) + \chi(5)(5)(4)}=\frac23.
\end{align*}
Now \eqref{eq:decompose} gives
\begin{align*}
  \chi &= a_2a_4\crit(2,4) + a_2a_5\crit(2,5) + a_3a_4\crit(3,4) + a_3a_5\crit(3,5)\\
    &= (2/9)\crit(2,4) + (4/9)\crit(2,5) + (1/9)\crit(3,4) + (2/9)\crit(3,5).
\end{align*}
\end{example}

It follows from \thref{lem:crit.seq,lem:decompose} that $m$-critical systems
have no nonzero fixed points. We show this together with some other simple
consequences of these lemmas:
\begin{prop}\thlabel{prop:critical.Phi}
  Suppose that $(\chi,h)$ is $m$-critical and that $\chi\neq\delta_1$. 
  Assume that $\chi$ has finite support,
  $h(\ell)$ is increasing, and that
  $h(\ell)\leq\ell$ for $\ell\geq 1$. Then the system's automaton distribution map $\Psi$
  satisfies $\Psi(x)<x$ for $x\in(0,1]$, and
  $\bigl(x-\Psi(x)\bigr)/\chi(0)$ is a convex combination of polynomials
  $\Bg{\ell}{r}$ for $\ell\in\tier(m)$ and $r\geq m+1$.
\end{prop}
\begin{proof}
  The statement about $\bigl(x-\Psi(x)\bigr)/\chi(0)$ is a direct consequence
  of \thref{lem:crit.seq,lem:decompose}:
  We use \thref{lem:decompose} to decompose $\chi$ as
  \begin{align*}
    \chi=\sum_{\ell_1,\ldots,\ell_m} a_{\ell_1,\ldots,\ell_m}\crit(\ell_1,\ldots,\ell_m),
  \end{align*}
  where the sum ranges over all $\ell_k\in\tier(k)$ for $k=1,\ldots,m$ and the coefficients
  $a_{\ell_1,\ldots,\ell_m}$ are nonnegative and sum to 1.
  Let $\Psi_{\ell_1,\ldots,\ell_m}$ denote the automaton distribution map for
  $\crit(\ell_1,\ldots,\ell_m)$, and observe that the automaton distribution map $\Psi$ of $(\chi,h)$
  also decomposes as
  \begin{align*}
    \Psi(x) = \sum_{\ell_1,\ldots,\ell_m} a_{\ell_1,\ldots,\ell_m}\Psi_{\ell_1,\ldots,\ell_m}(x).
  \end{align*}
  When $\ell_1\geq 2$, the expression
  $x-\Psi_{\ell_1,\ldots,\ell_m}(x)$
  is a linear combination of the polynomials
  \begin{align*}
    \Bg{\ell_m}{r},\quad r\in\{m+1,\ldots,\ell_m\}
  \end{align*}
  with nonnegative coefficients summing to $\crit(\ell_1,\ldots,\ell_m)\{0\}$
  by \thref{lem:crit.seq}\ref{i:crit.seq.b}. In fact, this holds when $\ell_1=1$
  as well, since then $x-\Psi_{\ell_1,\ldots,\ell_m}(x)=0$ and $\crit(\ell_1,\ldots,\ell_m)\{0\}=0$.
  Hence,
  \begin{align*}
    x - \Psi(x) = x - \sum_{\ell_1,\ldots,\ell_m} a_{\ell_1,\ldots,\ell_m}\Psi_{\ell_1,\ldots,\ell_m}(x)
      = \sum_{\ell_1,\ldots,\ell_m} a_{\ell_1,\ldots,\ell_m}\bigl(x-\Psi_{\ell_1,\ldots,\ell_m}(x)\bigr)
  \end{align*}
  is a linear combination of polynomials $\Bg{\ell_m}{r}$ for $\ell_m\in\tier(m)$
  and $m+1\leq r\leq \ell_m$, with nonnegative coefficients summing to
  \begin{align}\label{eq:chi0}
    \sum_{\ell_1,\ldots,\ell_m} a_{\ell_1,\ldots,\ell_m}\crit(\ell_1,\ldots,\ell_m)\{0\}=\chi(0).
  \end{align}
  
  Finally, we show that $\Psi(x)<x$ for $x\in(0,1]$. Since $x-\Psi(x)$ is a linear combination
  with nonnegative coefficients
  of polynomials $\Bg{\ell_m}{r}$ that are nonnegative on $[0,1]$, we have $\Psi(x)\leq x$ for $x\in[0,1]$.
  Since these polynomials are strictly positive for $x\in(0,1]$,
  the statement holds so long any of them are included in the decomposition of $x-\Psi(x)$.
  Since each summand on the left-hand side of \eqref{eq:chi0} is nonzero whenever $\ell_1\geq 2$
  by \thref{lem:crit.seq}\ref{i:crit.seq.fullsupport},
  this is true unless $a_{\ell_1,\ldots,\ell_m}=0$ for all $\ell_1\geq 2$. But since
  $\crit(\ell_1,\ldots,\ell_m)=\delta_1$ if $\ell_1=1$, this would imply that $\chi=\delta_1$.
\end{proof}

The following is a trivial but useful observation:
\begin{lemma}\thlabel{lem:fp.criterion}
  For any $x\in(0,1]$ and positive integer $r$ satisfying $h(r)\leq r$,
  the system $(\chi,h)$ has $x$ as a fixed
  point if and only if
  \begin{align}\label{eq:fp.criterion}
    \chi(r) = \frac{x - \sum_{\ell\neq r}\chi(\ell)\Bg{\ell}{h(\ell)}}{\Bg{r}{h(r)}}.
  \end{align}
\end{lemma}
\begin{proof}
  Since $\Bg{r}{h(r)}=\P[\Bin(r,x)\geq h(r)]$ is positive by our assumption that
  $x>0$ and $h(r)\leq r$, equation~\eqref{eq:fp.criterion} is equivalent to 
  \begin{align*}
    \chi(r)\P[\Bin(r,x)\geq h(r)] + \sum_{\ell\neq r}\chi(\ell)\P[\Bin(\ell,x)\geq h(\ell)] = x.
  \end{align*}
  That is, equation~\eqref{eq:fp.criterion} is equivalent to the statement $\Psi(x)=x$.
\end{proof}

Suppose that $(\chi,h)$ has maximum support $n$ and is $m$-critical,
and that $h(\ell)$ is increasing and satisfies $h(\ell)\leq\ell$ for $\ell\geq 1$.
By \thref{prop:critical.Phi}, this system has no fixed points other than $0$.
Now suppose that $r\geq n+1$ and $h(r)=m+1$, and that we wish to modify $\chi$
by shifting mass from $0$
onto $r$ to create a given fixed point. (There is no obvious reason we would want to do this,
but it turns out to be a key step in the proof of \thref{prop:truncations}.)
The previous lemma suggests that we can do so by setting $\chi(r)$
to make \eqref{eq:fp.criterion} hold.
But this might not be possible, since \eqref{eq:fp.criterion} may demand that $\chi(r)$ be
too small (i.e., negative) or too large (i.e., greater than $\chi(0)$).
The following result combines with \thref{lem:fp.criterion}
to show that neither of these occurs.

\begin{lemma}\thlabel{prop:varphi}
  Suppose that $(\chi,h)$ is $m$-critical and $\chi\neq\delta_1$.
  Assume that $\chi$ is
  supported on $\{0,\ldots,n\}$, that $h(\ell)$ is increasing,
  and that $h(\ell)\leq \ell$ for all $\ell\geq 1$. 
  Fix some integer $r\geq n+1$ and suppose that $h(r)=m+1$.
  Define $\varphi(x)$ for $x\in(0,1]$ by
  \begin{align}\label{eq:varphi}
    \varphi(x) = \frac{x - \sum_{\ell=1}^{n}\chi(\ell)\Bg{\ell}{h(\ell)}}{\Bg{r}{m+1}}.
  \end{align}
  Then $\varphi(x)\in(0,\chi(0)]$, and 
  $\varphi(x)$ is strictly increasing.
\end{lemma}
\begin{proof}
  Let $\Psi$ be the automaton distribution map of $(\chi,h)$, and note that
  the numerator on the right-hand side of \eqref{eq:varphi} is equal
  to $x-\Psi(x)$.
  By \thref{prop:critical.Phi}, this quantity is strictly positive, proving
  that $\varphi(x)>0$. \thref{prop:critical.Phi} also lets us
  express $x-\Psi(x)$ as
  a linear combination of polynomials $\Bg{\ell}{j}$ for $\ell<r$
  and $j\geq m+1$, with nonnegative coefficients.
  \thref{lem:binexpr} to follow shows that $\Bg{\ell}{j}/\Bg{r}{m+1}$
  is strictly increasing for all $\ell<r$ and $j\geq m+1$, proving that  
  $\varphi$ is strictly increasing.
  Finally, direct evaluation shows that $\varphi(1)=1-\sum_{\ell=1}^n\chi(\ell)=\chi(0)$,
  and since $\varphi$ is increasing we have $\varphi(x)\leq\chi(0)$ for $x\in(0,1]$.
\end{proof}

\newcommand{\nnn}{n}
\begin{lemma}\thlabel{lem:binexpr}
  Let $1\leq j\leq p$ and $1\leq k\leq r$.  
  If $p< r$ and $j\geq k$, or if $p\leq r$ and $j > k$, then
  \begin{align}\label{eq:binexpr}
    \frac{\Bg{p}{j}}{\Bg{r}{k}}
  \end{align}
  is strictly increasing in $x$ for $x\in(0,1]$.
\end{lemma}
\begin{proof}
  First we consider the case where $j=k$ and $r=p+1$.
  The event $\{\Bin(p+1,x)\geq j\}$ occurs if the first $p$ coin flips yield at least $j$ successes,
  or if they yield exactly $j-1$ successes and the final coin flip is a success. Hence
  \begin{align*}
    \P[\Bin(p+1,x)\geq j]=\P[\Bin(p,x)\geq j] + x\P[\Bin(p,x)=j-1],
  \end{align*}
  giving us
  \begin{align*}
    \frac{\P[\Bin(p,x)\geq j]}{\P[\Bin(p+1,x)\geq j]}
      &= 1 - \frac{x\P[\Bin(p,x)=j-1]}{\P[\Bin(p+1,x)\geq j]}\\
      &= 1 - \frac{\binom{p}{j-1}x^{j}(1-x)^{p-j+1}}{\sum_{\nnn=j}^{p+1}\binom{p+1}{\nnn}x^\nnn(1-x)^{p-\nnn+1}}\\
      &= 1 - \frac{\binom{p}{j-1}}{\sum_{\nnn=j}^{p+1}\binom{p+1}{\nnn}x^{\nnn-j}(1-x)^{j-\nnn}}.
  \end{align*}
  The expression $x^{\nnn-j}(1-x)^{j-\nnn}$ is increasing in $x$ for $\nnn\geq j$ and strictly increasing for $\nnn>j$.
  Since $j\leq p$, the sum contains at least one strictly increasing
  term. Hence this expression is strictly 
  increasing.
  
  Next, consider the case where $j = k+1$ and $p=r$. Here, we have
  \begin{align*}
    \frac{\P[\Bin(p,x)\geq k+1]}{\P[\Bin(p,x)\geq k]}
      &= \frac{\P[\Bin(p,x)\geq k]-\P[\Bin(p,x)=k]}{\P[\Bin(p,x)\geq k]}
      = 1 - \frac{\P[\Bin(p,x)=k]}{\P[\Bin(p,x)\geq k]}.
  \end{align*}
  Hence it suffices to show that $\P[\Bin(p,x)\geq k] / \P[\Bin(p,x)=k]$ is strictly increasing.
  We express this quantity as
  \begin{align*}
    \frac{\P[\Bin(p,x)\geq k]}{ \P[\Bin(p,x)=k]} &= \frac{\sum_{\nnn=k}^p\binom{p}{\nnn}x^\nnn(1-x)^{p-\nnn}}{\binom{p}{k}x^k(1-x)^{p-k}} = \sum_{\nnn=k}^p \frac{\binom{p}{\nnn}}{\binom{p}{k}} x^{\nnn-k}(1-x)^{k-\nnn}.
  \end{align*}
  As in the previous case, the expression $x^{\nnn-k}(1-x)^{\nnn-k}$ is increasing in $x$ for $\nnn\geq k$
  and is strictly increasing for $\nnn>k$, and at least one of the strictly increasing terms appears.
  
  Finally, the general case follows from the two special cases by expressing 
  \eqref{eq:binexpr} as a product of quotients considered in the special
  cases. For example,
  \begin{align*}
    \frac{\Bg{5}{3}}{\Bg{7}{2}} &= \frac{\Bg{5}{3}}{\Bg{6}{3}}\frac{\Bg{6}{3}}{\Bg{7}{3}}\frac{\Bg{7}{3}}{\Bg{7}{2}},
  \end{align*}
  and is hence the product of strictly increasing functions.
\end{proof}

\begin{proof}[Proof of \thref{prop:truncations}]
  Let $(\chi,h)$ be the $m$-truncation of an $m$-supercordant system,
  and let $\Psi$ be its automaton distribution map. We must
  show that $\Psi$ has a single fixed point on $(0,1]$.
  We observe that $(\chi,h)$ is $m$-supercordant itself, since by
  \thref{thm:derivatives} the first $m$ derivatives of $\Psi$ at $0$ are equal
  to those of the automaton distribution map of the original $m$-supercordant system (see
  \thref{rmk:only.tierm}).
  
  If $h(\ell)\leq\ell$ does not hold for all $\ell\geq 1$, 
  define a new system $(\widetilde\chi,\widetilde{h})$ where for $\ell\geq 1$,
  \begin{align*}
    \widetilde\chi(\ell) &= \begin{cases}
      \chi(\ell) & \text{if $h(\ell)\leq \ell$,}\\
      0 & \text{if $h(\ell)>\ell$,}
    \end{cases},&
    \widetilde{h}(\ell) &= \begin{cases}
      h(\ell) & \text{if $h(\ell)\leq \ell$,}\\
      \ell & \text{if $h(\ell)>\ell$,}
    \end{cases}
  \end{align*}
  with $\widetilde\chi(0)$ set to make $\widetilde\chi$ a probability measure.
  Note that $\widetilde{h}(\ell)$ is still increasing.
  Since $\Bg{\ell}{h(\ell)}=0$ when $h(\ell)>\ell$, 
  the systems $(\chi,h)$ and $(\widetilde\chi,\widetilde{h})$ have identical automaton distribution
  maps, and we can work with $(\widetilde\chi,\widetilde{h})$ in place of $(\chi,h)$.
  Thus we will assume without loss of generality that $h(\ell)\leq\ell$ for all $\ell\geq 1$.
  
  We first give a proof for the case that tier~$m$ of $(\chi,h)$ consists of a single value $r$.
  Since the system is supercordant, by Taylor approximation we have $\Psi(x)>x$ for $x\in(0,\epsilon)$ 
  for some sufficiently small $\epsilon>0$.
  Since $\Psi(1)\leq 1$, the graph of
  $\Psi$ eventually dips down below or onto the line $y=x$, and 
  hence $\Psi$ has some nonzero fixed point.
  Now we show it has at most one.
  Let $\bar\chi$ be the $(m-1)$-truncation of $\chi$.
  The system $(\bar\chi,h)$ is $(m-1)$-concordant by \thref{rmk:only.tierm}. Its maximum
  threshold is $m-1$ or less by definition of truncation. By \thref{lem:decompose}\ref{i:fulltiers},
  it is $(m-1)$-critical.
  Define a map $\varphi\colon(0,1]\to\mathbb{R}$ by
  \begin{align}\label{eq:varphi.redux}
    \varphi(x) &= \frac{x - \sum_{\ell=1}^{r-1}\chi(\ell)\Bg{\ell}{h(\ell)}}{\Bg{r}{m}}
      =  \frac{x - \sum_{\ell=1}^{r-1}\bar\chi(\ell)\Bg{\ell}{h(\ell)}}{\Bg{r}{m}}.
  \end{align}
  By \thref{lem:fp.criterion}, the nonzero fixed points of $(\chi,h)$ make up
  the set $\varphi^{-1}\bigl(\chi(r)\bigr)$.
  By \thref{prop:varphi} applied
  to $(\bar\chi,h)$, the function $\varphi(x)$ is strictly increasing. 
  Thus $\varphi^{-1}\bigl(\chi(r)\bigr)$ contains no more than one point,
  and $(\chi,h)$ has at most one nonzero fixed point.
  This shows that $(\chi,h)$ has exactly one nonzero fixed point $x_0$, with $\Psi(x)>x$
  for $x\in(0,x_0)$ and $\Psi(x)<x$ for $x\in(x_0,1]$.

  To extend the proof to the case where tier~$m$ contains more than one value, again
  assume that $(\chi,h)$ has maximum threshold~$m$ and is $m$-supercordant.
  As before, by Taylor approximation $\Psi(x)$ has at least one fixed point on $(0,1]$.
  Since the fixed points form a closed subset of $(0,1]$, there is a largest fixed point;
  call it $x_0$.
  Let $r$ be the smallest value in tier~$m$ of $(\chi,h)$.
  Our strategy now will be to construct a new system $(\widetilde\chi,h)$
  where all of tier~$m$ is concentrated on $r$. By the special case of the proposition
  we have already proven, this system has a unique nonzero fixed point. Then we will compare
  this system's automaton distribution map to $\Psi$ and show that $\Psi$ must also
  have a unique fixed point (see Figure~\ref{fig:chi.tilde}).
  
  To carry this out, first let $\bar\chi$ be the $(m-1)$-truncation of $\chi$ and define
  $\varphi$ by \eqref{eq:varphi.redux} again.
  Let $p=\varphi(x_0)$. Now we define a new probability measure $\widetilde\chi$ 
  supported on $\{0,\ldots,r\}$ by
  \begin{align*}
    \widetilde\chi(\ell) &= \begin{cases}
      \bar\chi(\ell) & \text{if $1\leq\ell\leq r-1$,}\\
      \bar\chi(0)-p & \text{if $\ell=0$,}\\
      p & \text{if $\ell=r$.}
    \end{cases}
  \end{align*}
  Note that this is a valid probability measure
  since $p\in(0,\bar\chi(0)]$ by \thref{prop:varphi} applied to $(\bar\chi,h)$.
  Now we consider the system $(\widetilde\chi,h)$. 
  By \thref{lem:fp.criterion}, it has $x_0$ as a fixed point.
  By the special case of this proposition we have already proven, the system $(\widetilde\chi,h)$
  has no other nonzero fixed points besides $x_0$, and $\widetilde\Psi(x)>x$ for
  $0<x<x_0$, where $\widetilde\Psi$ is the automaton distribution map of $(\widetilde\chi,h)$.
  
  We claim that
  $\Psi(x)\geq \widetilde\Psi(x)$ for $0<x<x_0$. 
  Indeed, comparing the two functions, we have
  \begin{align*}
    \Psi(x) - \widetilde\Psi(x) = \sum_{\ell=r}^{\infty}\chi(\ell)\Bg{\ell}{m} - p\Bg{r}{m}.
  \end{align*}
  Hence
  \begin{align*}
    \frac{\Psi(x) - \widetilde\Psi(x)}{\Bg{r}{m}} &= \chi(r) + \sum_{\ell=r+1}^{\infty}\chi(\ell)\frac{\Bg{\ell}{m}}{\Bg{r}{m}} - p.
  \end{align*}
  By \thref{lem:binexpr}, this expression is strictly decreasing in $x$. 
  It equals $0$ at $x=x_0$ since $\Psi(x)=\widetilde\Psi(x_0)=x_0$, and hence it
  is positive when $0<x<x_0$, establishing the claim.

  Since $\Psi(x)\geq\widetilde\Psi(x)$ for $0<x<x_0$, and we have already shown that
  $\widetilde\Psi(x)>x$ for $0<x<x_0$, the function $\Psi$ has no nonzero fixed points smaller than $x_0$.
  Since $x_0$ was taken to be the largest fixed point of $\Psi$, it is the only one.
\end{proof}
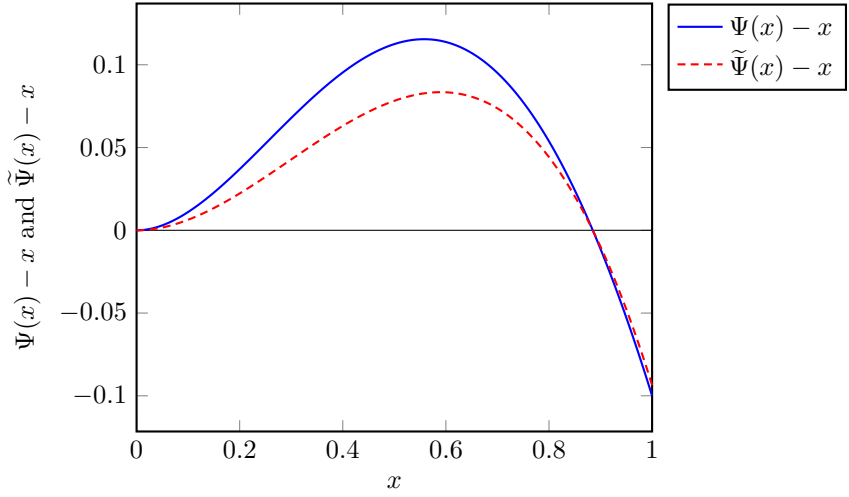
\begin{figure}
    \begin{tikzpicture}
    \begin{axis}[xmin=0, xmax=1,
                 legend pos=outer north east, thick, xlabel={$x$}, ylabel={$\Psi(x)-x$ and $\widetilde\Psi(x)-x$},
                 legend cell align=left,scaled ticks=false,y tick label style={/pgf/number format/fixed},
                 cycle list name=colorchange]
      \addplot[thin,black, forget plot, no markers] {0};
      \addplot+[no marks] table {psi.dat};
      \addplot+[no marks] table {psitilde.dat};
      \legend{$\Psi(x)-x$,$\widetilde\Psi(x)-x$}

    \end{axis}
    \end{tikzpicture}
    
  \caption{Let $\chi$ place vector of probabilities 
  $\bigl(\frac{1}{10},0,\frac12,\frac15,\frac15\bigr)$ on values $0$, $1$, $2$, $3$, $4$, and let
  $h(0)=h(2)=1$ and $h(3)=h(4)=2$.
  The system $(\chi,h)$ is $2$-supercordant; its automaton distribution map can be computed
  to be $\Psi(x)=x + \frac{13}{10}x^2 -2x^3 + \frac35 x^4$, with fixed point
  $x_0=(10-\sqrt{22})/6\approx .885$. The system $(\widetilde\chi,h)$ with the same
  fixed point but only a single value in tier~$2$ is found by computing
  $p=\varphi(x_0)\approx .406$, where $\varphi$ is given in \eqref{eq:varphi.redux},
  and then letting $\widetilde\chi$ place vector of probabilities $\bigl(\frac{1}{2}-p,0,\frac12,p,0\bigr)$
  on $0$, $1$, $2$, $3$, $4$. The automaton distribution map $\widetilde{\Psi}$ of $(\widetilde\chi,h)$
  is shown above together with $\Psi$.
  }
  \label{fig:chi.tilde}
\end{figure}

\section{Proofs of main theorems}

Recall the notation $n_t(v)$ for the number of children of $v$ in a rooted tree $t$
and the definition that $S$ is an admissible subtree of a rooted tree $T$ if 
$S$ contains the root of $T$ and $n_S(v)\geq h(n_T(v))$ for all vertices $v$ in $S$.
Also recall that $t(v)$ denotes the subtree of $t$ consisting of $v$ and all its descendants.

\begin{proof}[Proof of \thref{lem:highest.fp}]
  Let
  \begin{align*}
    \Tt_1=\{\text{$T$ contains an admissible subtree}\}.
  \end{align*}
  The largest fixed point $x_1$ of $(\chi,h)$ always has a corresponding
  interpretation \cite[Proposition~5.6]{JPS}. To prove that $\Tt_1$ is this interpretation,
  we first establish that $\Tt_1$ is an interpretation (i.e., it behaves consistently with 
  the threshold function $h$). Then, we show that for any interpretation $\Tt$,
  there exists an admissible subtree on the event $\Tt$, and hence
  $\Tt_1$ must have the largest probability of any interpretation.
  
  To show that $\Tt_1$ is an interpretation, we must show the following:
  Let $t$ be a tree with root $\rho$ that has $\ell$ children. Then we have
  $t\in\Tt_1$ if and only if $t(v)\in\Tt_1$ for at least $h(\ell)$ children $v$ of $\rho$.
  To prove this, first observe that if $s$ is an admissible subtree of $t$ containing a vertex~$v$,
  then $s(v)$ is an admissible subtree of $t(v)$.
  Now, suppose $t\in\Tt_1$. It thus contains an admissible subtree $s$.
  For each child $v\in s$ of $\rho$, we have $t(v)\in\Tt_1$ since $s(v)$ is an admissible subtree
  of $t(v)$. And since $s$ is admissible, $n_s(\rho)\geq h(\ell)$. Conversely, suppose there
  are at least $h(\ell)$ children $v$ of $\rho$ such that $t(v)\in\Tt_1$. Each subtree $t(v)$
  contains an admissible subtree. The concatenation of all of them together with $\rho$ is
  then admissible. This completes the proof that $\Tt_1$ is an interpretation.
  
  Now, let $\Tt$ be an arbitrary interpretation of $(\chi,h)$, and we argue that on the event
  $\Tt$ there exists an admissible subtree $S$ of $T$. To form $S$ when $\Tt$ occurs, let
  it include $\rho$. Then let it contain all children $v_1$ of $\rho$ for which $T(v_1)\in\Tt$,
  and then let it contain all children $v_2$ of these children for which $T(v_2)\in\Tt$, and so on.
  Since a tree~$t$ with root~$\rho$ is in $\Tt$ if and only if at least $h(n_t(\rho))$ of its root-child
  subtrees are in $\Tt$, the tree $S$ is admissible.
  
  To complete the proof, observe that $\GW{\chi}(\Tt_1)$ is a fixed point of $(\chi,h)$ since $\Tt_1$
  is an interpretation. For any interpretation $\Tt_1$, we have $\GW{\chi}(\Tt)\leq\GW{\chi}(\Tt_1)$, since
  $\Tt\subseteq\Tt_1$. Thus $\Tt_1$ must correspond to the largest interpretable fixed point,
  which is $x_1$.
\end{proof}

Next, we show that any admissible subtree contains a minimal admissible subtree within it.
This is akin to showing that a tree in which all vertices have at least two children contains
a subtree in which all vertices have exactly two children.

\begin{lemma}\thlabel{lem:minimal.admissible.subtree}
  Let $s$ be an admissible subtree of $t$ with respect to the threshold function $h$.
  Then there exists a subtree $s'\subseteq s$ that is also an admissible subtree of $t$
  for which $n_{s'}(v) = h(n_t(v))$ for all $v\in s'$.
\end{lemma}
\begin{proof}
  We construct $s'$ one level at a time.
  We start by including the root of $t$ in it.
  Now, suppose we have constructed it to level~$n$.
  Consider a level~$n$ vertex $v$ in $s'$.
  By admissibility it has at least $h(n_t(v))$ children in $s$.
  Arbitrarily choose exactly $h(n_t(v))$ of them to include in $s'$.
  Proceeding like this for all of the level~$n$ vertices in $s'$ and then continuing on
  to successive levels produces $s'\subseteq s$ that is an admissible subtree of $t$.
\end{proof}

\begin{proof}[Proof of \thref{thm:critical.event}]
  Let $r$ be the highest value in tier~$m$ of $(\chi,h)$.
  Let $\Tt_0$ be the event described in the statement of this theorem, that $T$
  contains an admissible subtree $S$ in which all but finitely many vertices $v$ satisfy
  $n_S(v)\leq m$. First, we show that $\Tt_0$ is an interpretation.
  The argument is mostly the same as for $\Tt_1$ being an interpretation in the proof
  of \thref{lem:highest.fp}.
  Let $t$ be a tree with root $\rho$ that has $\ell$ children. First, suppose
  that $t(v)\in\Tt_0$ holds for at least $h(\ell)$ of the children $v$ of $\rho$. For each
  such vertex $v$, the
  root-child subtree $t(v)$ thus contains an admissible subtree $s(v)$ for which all
  but finitely many vertices $u\in s(v)$ satisfy $n_{s(v)}(u)\leq m$.
  Combining each subtree $s(v)$ together with $\rho$ yields
  an admissible subtree of $t$ for which all but finitely many vertices have $m$ or fewer
  children, showing that
  $t\in\Tt_0$. Conversely, suppose that $t\in\Tt_0$, and let $s$ be an admissible subtree of $t$
  for which all but finitely many vertices $v\in s$ satisfy $n_s(v)\leq m$.
  In general, for any admissible subtree $s$ of $t$ and $v\in s$, the tree $s(v)$ is an admissible
  subtree of $t(v)$. And if all but finitely many vertices in $s$ have $m$ or fewer children,
  then the same is true for any subtree $s(v)$. Thus for any child $v$ of $\rho$ in $s$, 
  we have $t(v)\in\Tt_0$. By admissibility of $s$, we have $n_s(\rho)\geq h(\ell)$. 
  Hence $t(v)\in\Tt_0$ for at least $h(\ell)$ children $v$ of $\rho$.
   This completes the proof that $\Tt_0$ is an interpretation
  and its probability is therefore one of the fixed points of $\Psi$.
  
  Now we must determine which fixed point is associated with $\Tt_0$.
  Since $(\chi,h)$ is $m$-supercordant, by Taylor approximation we have $\Psi(x)>x$
  for $x\in(0,\epsilon)$ for a sufficiently small $\epsilon$. We cannot have $\Psi(x)>x$
  for all $x\in(0,1]$ since $\Psi(1)\leq 1$.
  Hence $\Psi$ has a smallest nonzero fixed point $x_0$. Our goal now
  is to show that $\GW{\chi}(\Tt_0)=x_0$.
  

  Let $\bar\chi$ be the $m$-truncation of $\chi$, and consider the system $(\bar\chi,h)$.
  Its automaton distribution map $\bar\Psi$ has a unique nonzero fixed point $\bar x_0$ by
  \thref{prop:truncations}. Directly from the definition of the automaton
  distribution map, we have $\bar\Psi(x)\leq\Psi(x)$. Also $(\bar\chi,h)$
  remains $m$-supercordant by \thref{rmk:only.tierm}, and hence
  the graph of $\bar\Psi$ is above the line $y=x$ near $0$.
  These last two facts prove that $\bar x_0\leq x_0$.
  Let $\barT\sim\GW{\bar\chi}$.
  Since $\bar x_0$ is the only nonzero fixed point of $\bar\Psi$,   by \thref{lem:highest.fp}
  the interpretation $\overline\Tt$ of $(\bar\chi,h)$ associated with $\bar x_0$
  is that $\barT$ contains an admissible subtree.

  Couple $\barT$ with $T$ by defining $\barT$ as the connected component of the root
  in the subgraph of $T$ consisting
  of the root together with each vertex whose parent $v$ satisfies $n_T(v)\leq r$.
  We claim that under this coupling, $\overline\Tt$ holds
  if and only if $T$ contains an admissible subtree $S$ made up entirely
  of vertices $v$ satisfying $n_S(v)\leq m$.
  To prove this, first observe that under this coupling,
  we have $n_{\barT}(v)=n_T(v)\1\{n_T(v)\leq r\}$  for $v\in\barT$.
  If $\barS$ is an admissible subtree of
  $\barT$, then it has no leaves by the positivity of $h$; hence
  $n_{\barT}(v)\geq n_{\barS}(v)\geq 1$ for $v\in\barS$.
  Therefore $n_\barT(v)=n_T(v)$ for all $v\in \barS$.
  This proves that $\barS$ is an admissible subtree not just of $\barT$
  but also of $T$. For every vertex $v\in \barS$ besides the root, 
  the parent $u$ of $v$ satisfies $n_\barT(v)=n_T(v)\leq r$. 
  Since $\barS$ has no leaves, every vertex in $\barS$ is the parent of some other vertex,
  and hence $n_T(v)\leq r$ for all $v\in\barS$.
  This proves that if $\overline\Tt$ holds, then $T$ contains an admissible subtree $\barS$ made up entirely
  of vertices $v$ satisfying $n_T(v)\leq r$. By \thref{lem:minimal.admissible.subtree},
  there exists a subtree $S'\subset\barS$ that is also an admissible subtree of $T$ and which
  satisfies $n_{S'}(v)=h(n_T(v))$ for all $v\in S'$. Since $n_T(v)\leq r$ and $h(r)=m$,
  we have $n_{S'}(v)\leq m$ for all $v\in S'$.  
  
  Conversely, suppose $T$ contains an admissible subtree $S$
  made up entirely
  of vertices $v$ satisfying $n_S(v)\leq m$.
  By admissibility, 
  all $v\in S$ also satisfy $n_S(v)\geq h(n_T(v))$. Hence $h(n_T(v))\leq m$,
  proving that $n_T(v)\leq r$. Thus $S$ is a subtree of $\barT$.
  We then have $n_\barT(v)=n_T(v)$ for all $v\in S$, showing that $S$ is an admissible
  subtree of $\barT$ and proving that $\overline\Tt$ holds. This proves
  that $\overline\Tt$ holds
  if and only if $T$ contains an admissible subtree made up entirely
  of vertices $v$ satisfying $n_T(v)\leq r$.  
  It is worth emphasizing that $\overline\Tt$ is not an interpretation of $(\chi,h)$:
  with $L$ the number of children of the root of $T$,
  we might have $T(v)\in\overline\Tt$ for at least $h(L)$ children $v$ of the root but have
  $T\notin\overline\Tt$ because $L>r$.
  
  Now let
  $\overline\Tt_n$ denote the event that
  $T$ contains an admissible subtree that from level~$n$ onward contains only vertices
  $v$ satisfying $n_T(v)\leq r$. 
  We have
  \begin{align}
    \overline\Tt_0 \subseteq \overline\Tt_1\subseteq\overline\Tt_2\subseteq\cdots,\qquad
    \text{and}\qquad\bigcup_{n=0}^{\infty}\overline\Tt_n = \Tt_0.\label{eq:union}
  \end{align}
  We claim that
  \begin{align*}
    \GW{\chi}\bigl(\overline\Tt_n\bigr) = \overbrace{\Psi\circ\cdots\circ\Psi}^{\text{$n$ times}}(\bar{x}_0).
  \end{align*}
  To see this, recall that $\bar x_0=\GW{\bar\chi}(\overline{\Tt})$. Since
  $\overline\Tt$ occurs if and only if $\overline\Tt_0$ occurs 
  under the coupling of $\barT$ and $T$, we have
  $\GW{\chi}(\overline\Tt_0)=\GW{\bar\chi}(\overline\Tt)=\bar x_0$.
  Thus $\Psi(\bar x_0)$
  is the probability that the root of $T$ has at least $h(L)$ children whose descendent subtrees
  satisfy event $\overline\Tt_0$, where $L$ is the number of children of the root. This event is 
  exactly $\overline\Tt_1$.
  Continuing in this way, the $n$-fold iteration of $\Psi$ applied to $\bar x_0$ is the probability of 
  $\overline\Tt_n$.
  
  From \eqref{eq:union},
  we can compute $\GW{\chi}(\Tt_0)$ by finding $\lim_{n\to\infty}\GW{\chi}\bigl(\overline\Tt_n\bigr)$.
  Because $x<\Psi(x)\leq x_0$ for $x\in(0,x_0]$, iteration of $\Psi(x)$ starting at any $x\in(0,x_0]$
  produces an increasing sequence converging to a value that must be a fixed point of $\Psi$
  by continuity of $\Psi$. This limit must therefore be $x_0$, the smallest
  nonzero fixed point.
  We therefore have $\GW{\chi}\bigl(\overline\Tt_n\bigr)\to x_0$, proving that $\GW{\chi}(\Tt_0)=x_0$.
\end{proof}

\section{Conclusions and remaining questions}\label{sec:questions}

In this paper, we give simple criteria for determining if a continuous phase transition
will occur at $(\chi,h)$ (\thref{thm:criticality}). When a continuous phase transition occurs,
we characterize the event undergoing the phase transition when it occurs in the most
natural way, with the graph of the automaton distribution map rising above the line $y=x$
as the phase transition occurs (\thref{thm:critical.event}). But some examples of continuous phase transitions
do not fit this description. In 
Figure~\ref{fig:phase.transition.classes.bad}, we give a family of child distributions 
in which two fixed points emerge from $0$ simultaneously
as the phase transition occurs. The event undergoing the phase transitions is associated with the
second of these, and \thref{thm:critical.event} does not apply. (In fact,
in the example in Figure~\ref{fig:phase.transition.classes.bad},
the interpretation associated with the second fixed point can be described as the existence
of an admissible subtree of $T\sim\GW{\chi_t}$ in which all but finitely many vertices $v$
satisfy $n_T(v)\leq 3$, along the lines as when \thref{thm:critical.event} applies. This holds
in this case because truncating
$\chi_t$ by shifting the mass on $6$ to $0$ yields a new system with a single nonzero fixed point,
as in \thref{prop:truncations}. But it is possible to tweak the example so this fails.) 
It also seems possible to construct examples along the same
lines as the one in Figure~\ref{fig:phase.transition.classes.bad} but with the automaton distribution map
repeatedly wiggling up and down along $y=x$ so that multiple interpretable
fixed points emerge from $0$ simultaneously. It is not clear to us how to describe the events
undergoing phase transitions in circumstances like these.

When a continuous phase transition occurs and $\GW{\chi_t}(\Tt)$ emerges from $0$
at $t=0$, it would be interesting to investigate the behavior of this probability.
For example, what behaviors can it show close to $t=0$? And how does this behavior compare to 
known or conjectured properties of phase transitions in statistical physics?

One might also want to generalize away from monotone automata and away from two-state automata
(see Section~\ref{subsec:automata.interpretations}).
For nonmonotone two-state automata, \thref{prop:interpretable.crit} fails but a more general
criterion \cite[Theorem~1.7]{JPS} still allows us to determine whether a given fixed point
has an interpretation or not. But the situation is very different; for example, the highest
fixed point is not always interpretable \cite[Examples~5.7, 5.8]{JPS}.
For multistate automata, only one direction of this criterion is proven, and the automaton
distribution map becomes a map from $\mathbb{R}^k$ to $\mathbb{R}^k$ and is generally harder to analyze.

The simplest case of \thref{thm:critical.event} to understand is for a system $(\chi,h)$
that is $1$-supercordant. Then the interpretation $\Tt_0$ associated with the smallest nonzero fixed point
is that $T\sim\GW{\chi}$ has an admissible subtree $S$ in which eventually all vertices $v$
have $n_S(v)=1$.  Thus, $\Tt_0$ is equivalent to the event that $T$ contains an admissible
subtree~$S$ such that the number of vertices at the $n$th level of $S$ is bounded in $n$.

Is there a description of $\Tt_0$ in terms of existence of an admissible subtree with specified growth
when $(\chi,h)$ is $m$-supercordant for $m\geq 2$? 
In this case, any admissible subtree must continue branching forever (i.e., its size at level~$n$
cannot remain bounded over all $n$).
Indeed, any admissible subtree whose size at level~$n$ remains bounded must have all but finitely
many of its vertices $v$ satisfying $h(n_T(v))=1$.
But restricting $T$ to such vertices yields a Galton--Watson tree with child distribution $\bar\chi$,
where $\bar\chi$ is the $1$-truncation of $\chi$. If $\bar\Psi$ is the automaton distribution map
of $(\bar\chi,h)$, then $\bar\Psi'(0)=\Psi'(0)=1$ by \thref{rmk:only.tierm}, and hence
$\bar\chi$ has expectation~$1$ by \eqref{eq:Psi'}, and this Galton--Watson tree is therefore critical.
Hence, for any $v\in T$, it cannot occur that $T(v)$ has an admissible subtree
consisting entirely of vertices $v$ satisfying $h(n_T(v))=1$, since the subtree
formed by these vertices is a critical Galton--Watson tree and is thus finite.
But we conjecture that $T$ has an admissible subtree of small growth on the event $\Tt_0$
(leaving it vague precisely what \emph{small growth} should mean),
while for all other interpretations $\Tt$ it occurs with positive probability
that all admissible subtrees have exponential growth. 

We present two examples to give some limited evidence for this conjecture.
Consider the example shown in Figure~\ref{fig:emergence}, where
$\chi_t$ places probability $1/2$ on $2$ and $1/6-t$ on $3$, and $h(2)=1$ and $h(3)=2$.
The system $(\chi_t,h)$ is $2$-supercordant for $t>0$, and it has a single nonzero fixed point $x_0=x_0(t)$.
The exact value of $x_0$ is not important in this example, though in this
case we can compute it to be $x_0=9t/(6t-1)$ by solving the equation $\Psi_t(x)=x$ directly.
By \thref{thm:critical.event} or \thref{lem:highest.fp}, this fixed point
has the interpretation $\Tt_0$ that $T\sim\GW{\chi_t}$ 
contains an admissible subtree.
We sketch a proof 
that $T$ has an admissible subtree of growth $e^{O(\sqrt{n})}$ on $\Tt_0$.

Let $L$ be the number of children of the root of $T$, and let $N$ be the number of these children
$v$ for which $T(v)\in\Tt_0$. Given $L=2$, which occurs with probability $1/2$, the probability
that $N=1$ is $\P[\Bin(2,x_0)=1]=2x_0(1-x_0)$ by self-similarity of $T$. Similarly, the probability that $N=2$
given $L=2$ is $x_0^2$. Hence
\begin{align*}
  \P[L=2,\, N=1\mid\Tt_0] = 1-x_0 \quad\qquad\text{and}\qquad\quad
  \P[L=2,\,N=2\mid\Tt_0] = \frac{x_0}{2}.
\end{align*}
Since $L\neq 0$ given $\Tt_0$, we have $\P[L=3\mid\Tt_0]=x_0/2$.

Now, consider the minimum number of vertices at level~$n$ over all admissible subtrees of $T$, 
given that $\Tt_0$ holds. We can construct a random variable $X_n$ with this distribution as follows.
Let $X_0=1$. Now, inductively define
\begin{align*}
  X_{n+1} = \begin{cases}
    X_n & \text{with prob.\ $\P[L=2,\, N=1\mid\Tt_0]=1-x_0$,}\\
    \min(X_n,\,X'_n) & \text{with prob.\ $\P[L=2,\, N=2\mid\Tt_0]=x_0/2$,}\\
    X_n+X'_n & \text{with prob.\ $\P[L=3,\, N=2\mid\Tt_0]$,}\\
    \min(X_n+X'_n,\,X_n+X''_n,\,X_n'+X_n'') & \text{with prob.\ $\P[L=3,\, N=3\mid\Tt_0]$,}
  \end{cases}
\end{align*}
where $X_n'$ and $X''_n$ are independent copies of $X_n$. 
We claim that $X_n$ is distributed as mentioned before. Indeed, this holds trivially for $X_0$. 
Proceeding inductively,
if $X_n$, $X_n'$, and $X_n''$ are thought of as the minimum number of vertices at level~$n$
in an admissible subtree of $T(v)$ for the three potential children of the root $v$, then
$X_{n+1}=X_n$ when $L=2$ and $N=1$, and $X_{n+1}=\min(X_n,\,X'_n)$ when
$L=2$ and $N=2$, and so on.

It is often difficult to analyze the growth of recursively defined distributions like these,
and we avoid doing so by comparing the growth of $X_n$ to a process
analyzed in \cite{AC} known as the min-plus binary tree. Particles of weight~$1$ start at the bottom
of a binary tree of depth~$n$. Each particle then moves up the tree. Each particle collides
with another one moving up the tree at each step, and with probability~$1/2$ either they merge
or the smaller particle annihilates the larger one. The size of the particle
arriving at the root has distribution given by the recursive contruction where $Y_0=1$ and then
\begin{align*}
  Y_{n+1} = \begin{cases}
    \min(Y_n,Y_n') & \text{with probability $1/2$,}\\
    Y_n+Y_n' & \text{with probability $1/2$},
  \end{cases}
\end{align*}
with $Y_n'$ an independent copy of $Y_n$.
One can show that $X_n$ is stochastically dominated by $Y_n$.
By \cite[Theorem~1]{AC}, we have $\P[Y_n\leq e^{\pi\sqrt{N/3}}]\to 1$ as $n$ tends to infinity.

Now, we give an example with multiple interpretable fixed points and demonstrate that on
the event associated with the higher one, the expected number of vertices in the smallest
admissible tree to level~$n$
can grow exponentially. Let
\begin{align*}
    \chi_t(\ell)=\begin{cases}
    1/2 + t & \text{for $\ell=2$,}\\
    1/2 - t & \text{for $\ell=3$,}\\
  \end{cases}\qquad\qquad \text{and} \qquad\qquad
  h(\ell)=\begin{cases}
      1 & \text{for $\ell=2$,}\\
      3 & \text{for $\ell=3$,}
  \end{cases}
\end{align*}
and let $\Psi_t$ be the automaton distribution map of $(\chi_t,h)$ (see
Figure~\ref{fig:exponential.tree}). For $0<t<1/6$, the map $\Psi_t$ has two nonzero fixed
points, $x_0=x_0(t)$ and $1$, both interpretable. The event associated with $x_0$
is that $T\sim\GW{\chi_t}$ has an admissible subtree $S$ that eventually consists
only of vertices $v$ with $n_S(v)=1$. Thus the number of vertices of $S$ at level~$n$ remains
bounded in $n$. The event associated with the fixed point $1$ is the set of all trees.
We argue that the smallest admissible subtree of $T$ may be large for $0<t<1/6$.
Indeed, let $X_n$ be the minimum number of vertices at level~$n$
over all admissible subtrees $T$. Then
\begin{align*}
  X_{n+1} \eqd \begin{cases}
    \min(X_n,\,X_n') & \text{with probability $1/2+t$,}\\
    X_n + X_n' + X_n'' & \text{with probability $1/2-t$,}
  \end{cases}
\end{align*}
and $\E X_{n+1} \geq 3(1/2-t)\E X_n$. Since $t<1/6$, we have $3(1/2-t)>1$, and hence
$E X_n$ grows exponentially.

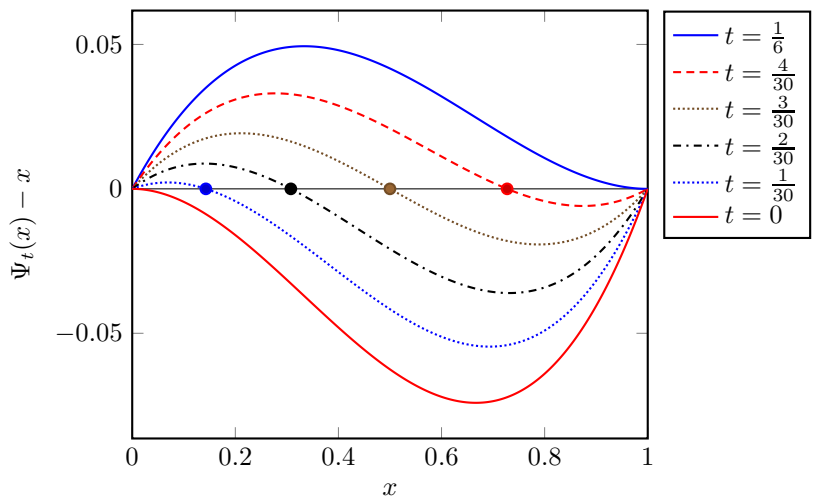
\begin{figure}
  \begin{tikzpicture}
    \begin{axis}[xmin=0, xmax=1,
                 legend pos=outer north east, thick, xlabel={$x$}, ylabel={$\Psi_t(x)-x$},
                 legend cell align=left,scaled ticks=false,y tick label style={/pgf/number format/fixed},
                 cycle list name=colorchange]
      \addplot[thin,black, forget plot, no markers] {0};
      \addplot+[no marks] table {eg5.dat};
      \addplot+[mark indices={183}] table {eg4.dat};
      \addplot+[mark indices={126}] table {eg3.dat};
      \addplot+[mark indices={78}] table {eg2.dat};
      \addplot+[mark indices={37}] table {eg1.dat};
      \addplot+[solid, no marks] table {eg0.dat};
      \legend{$t=\frac16$,$t=\frac{4}{30}$,$t=\frac{3}{30}$,$t=\frac{2}{30}$,$t=\frac{1}{30}$,$t=0$}

    \end{axis}
    
  \end{tikzpicture}
  \caption{Graphs of $\Psi_t(x)-x$, where $\Psi_t$ is the automaton distribution map
  of $(\chi_t,h)$ with  $\chi_t = \bigl(\frac{1}{2} + t\bigr)\delta_2 +
    \bigl(\frac12-t\bigr)\delta_3$, and $h(2)=1$ and $h(3)=3$.
    An interpretable fixed point emerges from $0$ as $t$ increases. The point $1$ is a fixed
    point in all examples, and for $t\geq 1/6$ it is the only fixed point.
    }\label{fig:exponential.tree}
\end{figure}

\section*{Acknowledgments}

We thank Itai Benjamini, who asked us if we could find criteria to distinguish continuous and first-order
phase transitions in the setting of Galton--Watson trees.
We also thank Fiona Skerman and Joel Spencer for their suggestions and comments.
We are grateful to a referee who improved this paper's exposition.

\bibliographystyle{amsalpha}
\bibliography{gw_transitions}

\end{document}